\documentclass[smallextended,natbib,runningheads]{svjour3}
\journalname{Annals of the Institute of Statistical Mathematics}
\smartqed 

\usepackage{amsmath}
\usepackage{amssymb}
\usepackage{graphicx}

\usepackage{bm}
\usepackage{amsfonts,enumitem}
\usepackage{multirow}
\usepackage[margin=1in]{geometry}

\newtheorem{assumption}{}

\begin{document}

  \title{Nonparametric inference for  additive models estimated via  simplified smooth backfitting  \thanks{The online version of this article contains supplementary material.}}
  
  \titlerunning{Inference with simplified smooth backfitting}
 
\author{Suneel Babu Chatla}

\institute{Department of Mathematical Sciences\\
The University of Texas at El Paso, 500 West University Avenue,  Texas 79968,  USA, \email{sbchatla@utep.edu}
}


\maketitle
\begin{abstract}
      We investigate hypothesis testing in nonparametric additive models estimated using simplified smooth backfitting (Huang and Yu, Journal of Computational and Graphical Statistics, \textbf{28(2)}, 386--400, 2019). Simplified smooth backfitting achieves oracle properties under regularity conditions and provides closed-form expressions of the estimators that are useful for deriving asymptotic properties. 
    We develop  a generalized likelihood ratio 
    (GLR) (Fan, Zhang and Zhang, Annals of statistics, \textbf{29(1)},153--193, 2001) and a loss function (LF) (Hong and Lee, Annals of Statistics, \textbf{41(3)}, 1166--1203, 2013) based testing framework for inference. 
    Under the null hypothesis, both the GLR and LF tests have asymptotically rescaled chi-squared distributions,  and both exhibit the Wilks phenomenon, which means the scaling constants and degrees of freedom are independent of nuisance parameters. These tests are asymptotically optimal in terms of rates of convergence for nonparametric hypothesis testing. Additionally, the bandwidths that are well-suited for model estimation may be useful for testing. We show that in additive models, the LF test is asymptotically more powerful than the GLR test. We use simulations to demonstrate the Wilks phenomenon and the power of these proposed GLR and LF tests, and a real example to illustrate their usefulness.
\end{abstract}
 

\textit{keywords:} Generalized likelihood ratio, Loss function, Hypothesis testing, Local polynomial regression, Wilks phenomenon

\section{Introduction}
   \label{sec:intro}

   Additive models are popular structural nonparametric regression models and have been widely studied in the literature \citep{friedman1981projection,hastie1990generalized}.  For a random sample $
   \{Y_i,X_{i1},\ldots,X_{id} \}_{i=1}^n$,  we consider the following  additive model: 
   \begin{align}
     Y_i &= \alpha_0+ \sum_{j=1}^d m_j(X_{ij})+ \epsilon_i, \quad  i=1,\ldots,n,
     \label{eqn:int-model}
     \end{align}	
     where $\{\epsilon_i, i=1, \dots, n\}$ is a sequence of i.i.d. random variables with mean zero and
     finite variance $\sigma^2$ and the additive components $m_j(\cdot)$'s are unknown smooth functions which are  identifiable subject to the constraints, $E[m_j(\cdot)]=0$ for  $j=1,\ldots,d$.
     
     Additive models do not suffer greatly from the curse of dimensionality because all of the unknown functions are one-dimensional. It is possible to estimate each additive component with the same asymptotic bias and variance of a theoretical estimate which uses the knowledge of other components. \citet{mam99existence} demonstrates that this oracle property holds true when smooth backfitting is used for estimation. Alternative estimation methods include marginal integration \citep{tjostheim1994nonparametric,linton1995kernel}, backfitting  \citep{buja1989linear,opsomer2000asymptotic}, penalized splines \citep{wood2017generalized} and simplified smooth backfitting \citep{huang2018}. In this study, we concentrate on simplified smooth backfitting. Simplified smooth backfitting, in addition to achieving oracle properties under regularity conditions, provides closed-form expressions of the estimators, which are convenient for deriving asymptotic properties. 
   
   After fitting an additive model, we are often interested in some hypothesis testing problems,  e.g. testing whether a specific additive component in (\ref{eqn:int-model}) is significant, or whether it may be replaced by a parametric form. For simple hypothesis problems such as component significance, the existing penalized estimation methods \citep{meier2009high,lian2012identification,horowitz2013penalized,lian2015separation} provide some quick answers. However, a hypothesis testing framework is necessary for rigorous treatment. 
   The theory for nonparametric hypothesis testing is well developed for univariate nonparametric models $(d=1)$ \citep{ingster1993asymptotically,hart2013nonparametric} but is somewhat limited for additive models. 
    \citet{hardle2004bootstrap} propose a bootstrap inference procedure for generalized semiparametric additive models which is based on marginal integration \citep{linton1995kernel}.  While their test statistic is asymptotically normal, convergence to normality is slow, so they propose a bootstrap approach to calculate its critical values. \cite{roca2005testing} study the testing of second-order interaction terms in generalized additive models. They propose a likelihood ratio test and an empirical-process test based on the deviance under the null and alternative hypotheses. The asymptotic distributions of their test statistics are unknown and are hard to derive, so they use a bootstrap procedure to approximate the null distribution. 
    
    Proceeding in this direction, \citet{fan2005nonparametric} propose a generalized likelihood ratio (GLR) test which is very simple to use. The GLR test compares the likelihood function under the null with the likelihood function under the alternative. \citet{fan2005nonparametric} derive the asymptotic properties of the GLR test statistic using  classical backfitting \citep{opsomer2000asymptotic} for model estimation. It is known that backfitting does not achieve oracle bias when the covariates are correlated. Moreover, the estimators of backfitting do not have closed-form expressions. Regardless of these drawbacks of backfitting, \citet{fan2005nonparametric} show that the GLR test 
    exhibits Wilks phenomenon, which means that the null distribution is
   independent of the nuisance parameters -- a much-desired property for likelihood ratio tests. However, their method is still limited in practice because of the disadvantages of backfitting. Better alternatives for model estimation include smooth backfitting, for which the properties of the GLR test still need to be investigated. This motivates us to study the properties of GLR test using simplified smooth backfitting for estimation. 

    While the GLR test has some appealing features and has been widely used in practice, it is still a nonparametric pseudo test because of the parametric assumptions on error distribution. Another promising alternative is loss function (LF) based testing framework \citep{hong2013loss} which is available for univariate nonparametric models (Model (\ref{eqn:int-model}) with $d=1$). A loss function test compares the models under null and alternative by specifying a penalty for their difference. Many times, this is more relevant to decision-making under uncertainty because it provides the flexibility of choosing a loss function that mimics the objective of the decision-maker. \citet{hong2013loss} show that  LF test is asymptotically more powerful than GLR test  in terms of Pitman's efficiency criterion and possesses both optimal power and Wilks properties. Moreover, all admissible loss functions are asymptotically equally efficient under a general class of local alternatives. In spite of all these advantages,  the properties of  LF test still need to be investigated for nonparametric additive models $(d>1)$. To fill this void, we propose a LF test for additive model (\ref{eqn:int-model}) and derive its asymptotic properties. More recently, although in a different context, \citet{mammen2022backfitting} proposed a backfitting test. Their test compares the nonparametric estimators obtained from smooth backftting in the  $L_2$ norm. Using simulations they show that the backfitting test provides very good performance in finite samples. It is worth mentioning that, the proposed LF test in the study takes a similar form asymptotically.
    
    The main contributions from this study are as follows. We develop  GLR and LF based hypothesis testing frameworks for nonparametric additive model (\ref{eqn:int-model})  using simplified smooth backfitting \citep{huang2018} for estimation. In Theorems \ref{thm:glr1} and \ref{thm:loss1}, we show that both these test statistics follow a rescaled chi-square distribution asymptotically and achieve Wilks phenomenon. Unlike the GLR test in \citet{fan2005nonparametric}, the proposed GLR and LF tests do not require undersmoothing to achieve Wilks phenomenon and the bandwidths that were well-suited for model estimation might also be useful for testing. We also construct new $F$ type of tests for additive models  and establish the connections between GLR, LF and F-test statistics. Theorem \ref{thm:opt-test} shows that GLR and LF test statistics achieve the optimal rate of convergence for nonparametric testing, $n^{-2\eta/4\eta+1}$ where $\eta=2(p+1)$ and $p$ is the order of local polynomial, according to \citet{ingster1993asymptotically}. Furthermore, in Theorem \ref{thm:are} we show that LF test is asymptotically more powerful than  GLR test. Using simulations we validate our theoretical findings and illustrate that both  GLR and LF tests are robust to error distributions to some extent.

    The remainder of the paper is organized as follows. In Section \ref{sec:back}, we introduce smoother matrices which are required for simplified smooth backfitting and outline the estimation algorithm. In Section \ref{sec:hypo-test}, we formulate  GLR and LF test statistics for nonparametric additive model. We derive the asymptotic null distributions for both test statistics and discuss their optimal power properties in Section \ref{sec:asym}. In Section \ref{sec:numeric}, we evaluate the finite sample performances of both GLR  and LF tests using a simulation study and a real example. We include proofs and  additional numerical results in  Supplementary Material.

  \section{Simplified Smooth Backfitting} \label{sec:back}
  In this section, we give a brief introduction to  local polynomial regression \citep{fan1996local} and describe simplified smooth backfitting algorithm, which includes smoother matrices, $\bm{H}_p^*$, of \citet{hua08analysis}.

\subsection{Smoother Matrix}
 Suppose $(Z_{i},Y_i)$, $i=1,\ldots,n$, are $n$ independent observations generated from the following model
\begin{align}
Y = m(Z) + \epsilon,
\label{eqn:back-umodel}
\end{align}
where $Y$ is a continuous response variable,  $Z$ is a continuous explanatory variable and  $\epsilon$ denotes an error term with mean zero and finite variance. We choose local polynomial modeling approach \citep{fan1996local}. To estimate the conditional mean $E(Y|Z=z)$ at a grid point $z$, it considers a $p$th order Taylor expansion $m(z)+ m^{(1)}(z)(Z-z) + \ldots+ m^{(p)}(z)(Z-z)^p/p!$, for $Z$ in a neighborhood of $z$.  

Let $\bm{Z}_{z} =[\bm{1} \quad \bm{z}_{1} \cdots \bm{z}_{p}]_{n \times (p+1)}$ be a design matrix with $\bm{1}=(1,\ldots,1)^T$ of length $n$ and $\bm{z}_{r}=((Z_{1}-z)^r, \ldots, (Z_{n}-z)^r)^T$ for $r=1,\ldots,p$. Let $\bm{W}_{z}=\text{diag}\{K_{h}(Z_{1}-z),\ldots,K_{h}(Z_{n}-z)\}$ be a weight matrix with $K(\cdot)$ as a nonnegative and symmetric probability density function, and $K_{h}(\cdot)=K(\cdot/h)/h$ where $h$ is a bandwidth.
The local polynomial approach estimates  $\bm{\beta}=(\beta_0,\ldots,\beta_p)^T$, where $\beta_r=m^{(r)}/r!$, $r=0,1,\ldots,p$, as
\begin{align}
\underset{\bm{\beta}}{\text{min}} ~ \frac{1}{n} \sum_{i=1}^n \left(Y_i-\sum_{r=0}^p\beta_r\left(Z_{i}-z\right)^r\right)^2 K_{h}(Z_{i}-z)
&=  \underset{\bm{\beta}}{\text{min}} ~ \frac{1}{n} \left(\bm{y}-\bm{Z}_{z}\bm{\beta}\right)^T \bm{W}_{z}\left(\bm{y}-\bm{Z}_{z}\bm{\beta}\right),
\label{eqn:back-uob}
\end{align}
where $\bm{y}=(Y_1,\ldots,Y_n)^T$.  Let $\widehat{\bm{\beta}}=(\widehat{\beta}_0,\ldots,\widehat{\beta}_p)^T = (\bm{Z}_{z}^{T}\bm{W}_{z}\bm{Z}_{z})^{-1}\bm{Z}_{z}^{T}\bm{W}_{z}\bm{y}$, denote the solution vector to (\ref{eqn:back-uob}) and the dependence of $\widehat{\bm{\beta}}$  on $z$ is suppressed for convenience in notation. 

Suppose the support of $Z$ is $[0,1]$. Let 
\begin{align}
  K_h(u,v) &= \frac{K_h(u-v)}{\int K_h(w-v)dw} I(u,v \in [0,1]), 
\end{align}
be the boundary corrected kernel function defined in \citet{mam99existence}. It is easy to note that $\int K_h(u,v)du=1$ for a fixed $v$. The smoother matrix $\bm{H}_{p}^*$ in \citet{hua08analysis} is based on integrating local least squares errors (\ref{eqn:back-uob})
\begin{align}
\frac{1}{n} \int \sum_{i=1}^n \left(Y_i-\sum_{r=0}^p\widehat{\beta}_r\left(Z_{i}-z\right)^r\right)^2 K_{h}(Z_{i},z) dz
=   \frac{1}{n} \bm{y}^T\left(\bm{I}-\bm{H}_{p}^*\right)\bm{y},
\label{eqn:back-iuobj}
\end{align}
where $\bm{I}$ is an $n$- dimensional identity matrix and 
\begin{align}
\bm{H}_{p}^* =\int \bm{W}_{z} \bm{Z}_{z}(\bm{Z}_{z}^{T}\bm{W}_{z}\bm{Z}_{z})^{-1}\bm{Z}_{z}^{T}\bm{W}_{z} dz,
\label{eqn:back-hdef1}
\end{align}
is a smoother matrix in which  the integration is taken  element by element.  
Consequently, we may use $\bm{H}_{p}^*\bm{y}$  as a fitted value for $\bm{y}$ and its $i$th element, $\widehat{m}(Z_{i}) =\bm{e}_i^T\bm{H}_{p}^*\bm{y}$ takes the following form, for $i=1,\ldots,n,$ 
\begin{align}
\int \left(\widehat{\beta}_0+\widehat{\beta}_1 (Z_{i}-z) + \ldots+ \widehat{\beta}_p\left(Z_{i}-z\right)^p \right) K_{h}(Z_{i},z)dz, \label{eqn:back-fit}
\end{align}
where $\bm{e}_i$ is the unit vector with $1$ as the $i$th element.  The estimator (\ref{eqn:back-fit}) involves double smoothing as it combines the fitted polynomials around $Z_{i}$. Therefore, at interior points $[2h, 1-2h]$, the estimator $\widehat{m}(Z_{i})$ achieves bias reduction. While the  bias of traditional local polynomial estimator $\widehat{\beta}_0$ is of order $h^{(p+1)}$ for odd $p$, the bias of  $\widehat{m}(Z_{i})$ is of order $h^{2(p+1)}$ for $p=0,1,2,3$.  In Section \ref{subsubsec:simple-sb}, we define simplified  smooth backfitting estimators for additive model (\ref{eqn:int-model}) analogous to (\ref{eqn:back-fit}).  \citet{hua08analysis} and \citet{hua14local} discuss the properties of $\bm{H}_{p}^*$ and they show that it is symmetric, asymptotically idempotent and asymptotically a projection matrix. Moreover, it is nonnegative definite and shrinking.

\subsection{Estimation}
\label{subsubsec:simple-sb}

\citet{huang2018}'s simplified smooth backfitting algorithm is analogous to the classical backfitting algorithm of \citet{buja1989linear} and \citet{hastie1990generalized} in terms of component updates. The key difference is that it uses the univariate matrices $\bm{H}_{p}^*$ in (\ref{eqn:back-hdef1}) as smoothers in backfitting algorithm. 

Let $\bm{m}_j=(m_j(X_{1j}),\ldots, m_j(X_{nj}))^T$ and $\bm{x}_j=(X_{1j},\ldots,X_{nj})^T$ for $j=1,\ldots,d$. Let  $\mathbb{X}_j=[\bm{1} ~ \bm{x}_j ~\cdots ~ \bm{x}_j^{p_j}]$ for $j=1,\ldots,d$, and $\mathbb{X}=[\bm{1} ~ \bm{x}_1 ~ \cdots ~\bm{x}_d ~ \cdots ~ \bm{x}_1^{p_1}~ \cdots ~ \bm{x}_d^{p_d}]$, where $\bm{1}$ is the vector of ones. Let $\mathbb{X}^{[-0]}=[\bm{x}_1 ~ \cdots ~ \bm{x}_d ~ \cdots ~ \bm{x}_1^{p_1}~ \cdots ~ \bm{x}_d^{p_d}]$ which is same as $\mathbb{X}$ without the column of ones.  For any matrix $\bm{A}$, define $\bm{A}^{\perp}=\bm{I}-\bm{A}$ and $\bm{P}_{\bm{A}}=\bm{A}(\bm{A}^T\bm{A})^{-1}\bm{A}^T$.
Suppose $\bm{H}_{p_j,j}^*$ is an univariate smoother matrix defined as in (\ref{eqn:back-hdef1}) for covariate $\bm{x}_j$ with $p_j$th order local polynomial approximation and bandwidth $h_j$ for $j=1,\ldots,d$. 

 We introduce the following spaces before stating the simplified smooth backfitting algorithm for model (\ref{eqn:int-model}).  Let $\mathcal{M}_1(\bm{H}_{p_j,j}^*)$ be a space spanned by the
eigenvectors of $\bm{H}_{p_j,j}^*$ with eigenvalue 1. It includes polynomials of $\bm{x}_j$ up to $p_j$th order because $\bm{H}_{p_j,j}^* \bm{x}_j^k=\bm{x}_j^k$, $k=0,1\ldots,p_j$, and $j=1,\ldots,d$.  Suppose $\bm{G}$ is an orthogonal projection onto the space $\mathcal{M}_1(\bm{H}_{p_1,1}^*)+ \cdots + \mathcal{M}_1(\bm{H}_{p_d,d}^*)$ and $\bm{G}_j$ is an orthogonal projection onto the space $\mathcal{M}_1(\bm{H}_{p_j,j}^*)$, $j=1,\ldots,d$. Then,
\begin{align}
  \bm{G} = \bm{P}_{\mathbb{X}} = \bm{P}_{\bm{1}} + \bm{P}_{\bm{P}_{\bm{1}}^{\perp} \mathbb{X}^{[-0]}}, \qquad \bm{G}_j = \bm{P}_{\mathbb{X}_j}, \label{eqn:back-gmat-def}
\end{align}
where $\bm{P}_{\mathbb{X}}=\mathbb{X}(\mathbb{X}^T\mathbb{X})^{-1}\mathbb{X}$ and $\bm{P}_{\bm{1}}$, $\bm{P}_{\mathbb{X}_j}$ and $\bm{P}_{\bm{P}_{\bm{1}}^{\perp} \mathbb{X}^{[-0]}}$ are defined similarly. Since the modified smoothers $\bm{H}_{p_j,j}^*- \bm{G}_j$, $j=1,\ldots,d$, have eigenvalues in  $[0,1)$,  by Proposition 3 in \citet{buja1989linear}, we obtain  closed form expressions for backfitting estimators. For illustration, we plot the eigenvalues of the smoother and the modified smoother using local constant ($p=0$) and local linear terms ($p=1$) in Figure \ref{fig:back-ssb-eig}. While the local constant smoother has one eigenvalue equal to 1, the local linear smoother has two eigenvalues that are equal to 1. The modified smoother has eigenvalues in $[0,1)$ for bandwidths that are not too small.  

The simplified smooth backfitting algorithm with modified smoothers $\bm{H}_{p_j,j}^*- \bm{G}_j$ for $j=1,\ldots,d$, and $p_j =0,1, 2, 3$,  is stated as follows:
\begin{enumerate}
	\item Initialize: $\bm{m}_j^*=\bm{m}_j^{(0)}$ with $\bm{m}_j^{(0)}$ in the space of $\mathcal{M}(\bm{H}_{p_j,j}^*- \bm{G}_j)$, $j=1,\ldots,d$.
  \item Cycle:  $\bm{m}_j^{*^{new}}=\bm{H}_{p_j,j}^*\left(\bm{G}^{\perp}\bm{y}-\sum_{l<j}\bm{m}_l^{*^{new}}-\sum_{l>j}\bm{m}_l^{*^{old}}\right)$, $j=1,\ldots,d$, since 
  $ \bm{G}_j\bm{m}_l^{*^{old}}=\bm{G}_j\bm{m}_l^{*^{new}}=\bm{0}$   and $\bm{G}_j \bm{G}^{\perp}=\bm{0}$.
	\item Continue step 2 until the individual functions do not change. The final estimator for the overall fit is $\bm{G}\bm{y} + \widehat{\bm{m}}_1^*+ \ldots + \widehat{\bm{m}}_d^*$.
\end{enumerate} 
\begin{figure}
  \centering
  \includegraphics[scale=0.35]{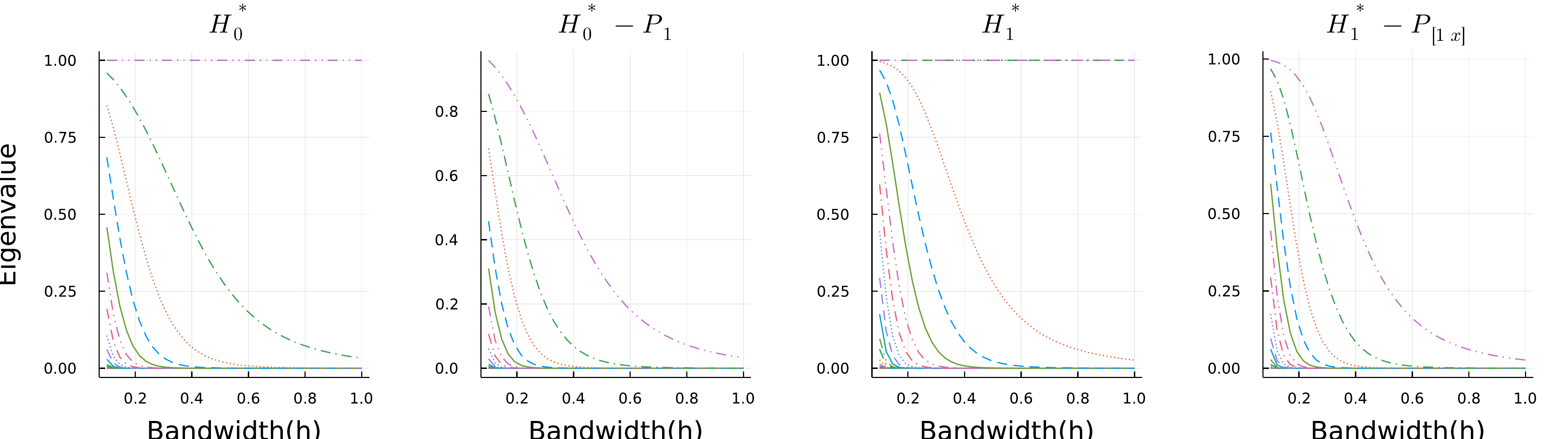}
  \caption{Eigenvalues of smoother $\bm{H}_{p_j}^*$ and modified smoother $\bm{H}_{p_j}^*-\bm{G}_j$ for different values of bandwidths and for $p_j=0,1$. Here $\bm{G}_j=\bm{P}_1$ and $\bm{G}_j=\bm{P}_{[\bm{1} ~ \bm{x}_j]}$ for $p_j=0$ and 1, respectively. }
  \label{fig:back-ssb-eig}
\end{figure}

Furthermore, we can write
\begin{align*}
  \bm{G}\bm{y} &=  \left(\bm{P}_{\bm{1}} + \bm{P}_{\bm{P}_{\bm{1}}^{\perp} \mathbb{X}^{[-0]}} \right) \bm{y} = \widehat{\alpha}_0 \bm{1} + \widehat{\bm{g}}_1+\ldots+\widehat{\bm{g}}_d,  
\end{align*}
for $\widehat{\bm{g}}_j=(\widehat{g}_{1j},\ldots,\widehat{g}_{nj})^T$ such that $\sum_{i=1}^n \widehat{g}_{ij}=0$, $j=1,\ldots,d$. Then the  final estimator for $j$th additive component is $\widehat{\bm{m}}_j= (\widehat{m}_j(X_{1j}),\ldots, \widehat{m}_j(X_{nj}))^T= \widehat{\bm{g}}_j + \widehat{\bm{m}}_j^*$, for $j=1,\ldots,d$. Since $\bm{G}_j\widehat{\bm{m}}_j^*=0$  it follows that $\sum_{i=1}^n \widehat{m}_j(X_{ij})=0$ for $j=1,\ldots,d$. Consequently, the estimators $\widehat{\bm{m}}_j$'s are identifiable.

Since the smoothers $\bm{H}_{p_j,j}^*$, $j=1,\ldots,d$, are symmetric and shrinking, using the results in \citet{buja1989linear}, \citet{huang2018} show that the above algorithm converges. We provide their results in the following:
\begin{itemize}
  \item It follows from Theorem 2 of \citet{buja1989linear} that the normal equations 
  \begin{align}
      \begin{pmatrix}
        \bm{I} & \bm{H}_{p_1,1}^* & \bm{H}_{p_1,1}^* & \cdots & \bm{H}_{p_1,1}^* \\
        \bm{H}_{p_2,2}^* & \bm{I} & \bm{H}_{p_2,2}^* & \cdots & \bm{H}_{p_2,2}^* \\
        \vdots & \vdots & \vdots & \ddots & \vdots \\
        \bm{H}_{p_d,d}^* & \bm{H}_{p_d,d}^* & \bm{H}_{p_d,d}^* & \cdots & \bm{I} 
      \end{pmatrix} \begin{pmatrix}
        \bm{m}_1\\
        \bm{m}_2\\
        \vdots \\
        \bm{m}_d
      \end{pmatrix}
      =\begin{pmatrix}
        \bm{H}_{p_1,1}^* \bm{y}\\
        \bm{H}_{p_2,2}^* \bm{y}\\
        \vdots\\
        \bm{H}_{p_d,d}^* \bm{y}
      \end{pmatrix}
      \label{eqn:back-ssb-normal}
  \end{align}
  are consistent for every $\bm{y}$.
  \item Based on Theorem 9 of \citet{buja1989linear}, the backfitting algorithm converges to some solution of the normal equations (\ref{eqn:back-ssb-normal}). 
  \item The solution is unique unless there is an exact concurvity which happens when there is a linear dependence among the eigenspaces corresponding to eigenvalue 1 of the $\bm{H}_{p_j,j}$'s. 
\end{itemize}

We now obtain explicit expressions for the estimators $\widehat{\bm{m}}_j^*$, $j=1,\ldots,d$. Let $\bm{A}_j=(\bm{I}-(\bm{H}_{p_j,j}^*- \bm{G}_j))^{-1}(\bm{H}_{p_j,j}^*- \bm{G}_j)$ and  $\bm{A}=\sum_{j=1}^d\bm{A}_j$. By Proposition 3 in \citet{buja1989linear}, we obtain $\widehat{\bm{m}}_j^* = \bm{A}_j \left(\bm{I}+\bm{A}\right)^{-1}\bm{G}^{\perp}\bm{y}$. While these expressions provide estimators without requiring an iterative procedure, we still favor the backfitting algorithm because of its numerical stability. Furthermore, the backfitting approach is computationally simpler. Assume that a certain number of iterations is sufficient for convergence. In terms of computations, the explicit expressions cost $O(n^3p)$ operations whereas the  backfitting algorithm only costs $O(np)$ operations \citep{hastie1990generalized}. However, this might not be a concern for small sample sizes. In  Figure \ref{fig:back-explicit}, we provide a comparison of the estimated functions $\widehat{m}_j^*$ using  backfitting and explicit expressions for the model considered in Section \ref{sec:simul}. Both methods provide similar results. The computation times of backfitting and explicit expressions across different sample sizes are presented in Table \ref{tab:ctime-back-exp}. For small sample sizes, both approaches took about the same amount of time. However,  solving explicit expressions is computationally costly than backfitting for large sample sizes. More information on the computer facilities can be found in Section \ref{sec:simul}.
\begin{figure}
  \centering
  \includegraphics[scale=0.35]{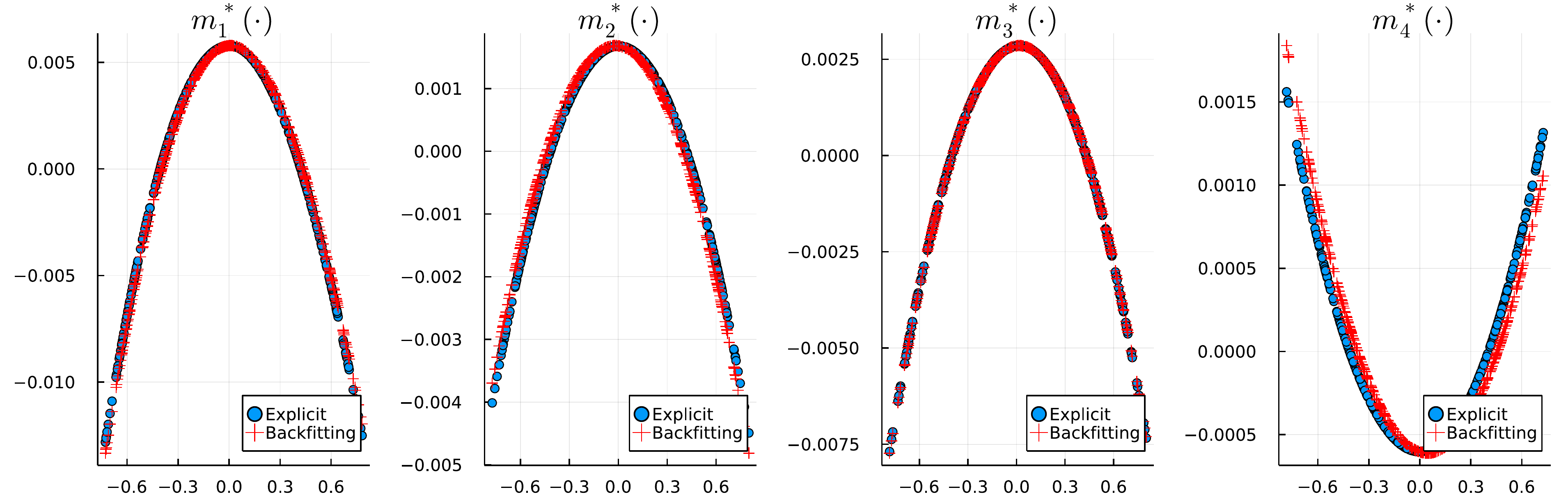}
  \caption{The estimated additive functions (nonparametric part) for model (\ref{eq:sim-model}) (shown later in Section \ref{sec:simul}) using the explicit expressions and the backfitting algorithm (19 iterations). Here $n=400$ and optimal bandwidths $(\widehat{h}_{1,opt},\widehat{h}_{2,opt},\widehat{h}_{3,opt},\widehat{h}_{4,opt})=(0.74,1.2,1.15,1.16)$.}
  \label{fig:back-explicit}
\end{figure}
\begin{table}
  \centering
\begin{tabular}{lll}
  \hline
  $n$  & Backfitting(ms)  & Explicit(ms) \\
    \hline
  200 &  664.8 & 682.6 \\
  400 &  2414  & 2514 \\
  800 &  8961  &   8650 \\
    1600 & 40717 & 40299 \\
    3000 & 169575 & 194010\\
      6000 & 856743 & 1203930\\
 \hline
  \end{tabular}
  \caption{Comparison of computation times (in milliseconds)  for model (\ref{eq:sim-model}) (shown later in Section \ref{sec:simul}) using backfitting algorithm and explicit expressions.}
  \label{tab:ctime-back-exp}
\end{table}

  At the convergence of simplified smooth backfitting algorithm, we obtain smooth backfitting estimates (or estimates at grid points) $\widehat{\beta}_{jr}$   by performing a local polynomial regression of $\bm{x}_j$ on partial residual $(\bm{G}^{\perp}\bm{y}-\sum_{l<j}\widehat{\bm{m}}_l-\sum_{l>j}\widehat{\bm{m}}_l)$.
Formally,
\begin{align}
  \widehat{\beta}_{jr} &= \bm{e}_r^T(\bm{X}_{x_j}^{j^T} \bm{W}_{x_j}^j \bm{X}_{x_j}^j)^{-1} \bm{X}_{x_j}^{j^T} \bm{W}_{x_j}^j(\bm{G}^{\perp}\bm{y}-\sum_{l<j}\widehat{\bm{m}}_l-\sum_{l>j}\widehat{\bm{m}}_l),
  \label{eqn:beta-estr}
\end{align}  
for $1 \le j \le d$, $0 \le r \le p_j$, where $\bm{e}_r$ is a unit vector with 1 at the $r$th position and $\bm{W}_{x_j}^j$ and $\bm{X}_{x_j}^j$ are defined similar to $\bm{W}_{z}$ and $\bm{Z}_{z}$ in (\ref{eqn:back-uob}).

\citet{huang2018} discuss the properties of estimators $\widehat{\bm{m}}_j^*$ and $\widehat{\beta}_{j0}$ for $j=1,\ldots,d$. They show that,  estimator $\widehat{\bm{m}}_j^*$ achieves asymptotic bias of order $\sum_{j=1}^d h_j^{2(p_j+1)}$, for $p_j=0,1,2,3$, in the interior range $[2h_j,1-2h_j]$ for $j=1,\ldots,d$. Similarly, asymptotic bias of $\widehat{\beta}_{j0}$ is of order $h_j^{(p_j+1)}+\sum_{k \neq j}^d h_k^{2(p_k+1)}$ if $p_j=1$ or $3$, and is of order $h_j^{(p_j+2)}+\sum_{k \neq j}^d h_k^{2(p_k+1)}$ if $p_j=0$ or $2$,  for $j=1,\ldots,d$, in the interior range. 

\section{Proposed Test Statistics} \label{sec:hypo-test}
In this section, we define both  GLR  and LF test statistics for  Model (\ref{eqn:int-model}) which are computed using the simplified smooth backfitting in Section \ref{sec:back}. For simplicity in presentation, we consider the following  simple hypothesis testing problem:
\begin{align}
H_0: m_d(x_d)=0 \quad \text{vs.} \qquad   H_1: m_d(x_d) \ne 0, 
\label{eqn:h1}
\end{align}
which tests whether the $d$th predictor makes any significant contribution to the dependent variable. This testing problem is a nonparametric null versus a nonparametric alternative. It is also possible to choose more complicated hypothesis testing problems such as composite hypotheses, and nonparametric null versus parametric alternatives. We discuss some of these in our numerical results in Section \ref{sec:numeric}.

We now introduce some matrices which will be used in our asymptotic results. Let $\bm{G}_{[-d]}=\bm{P}_{\mathbb{X}^{[-d]}}$ where $\mathbb{X}^{[-d]}=[\bm{1} ~ \bm{x}_1 ~ \cdots ~ \bm{x}_{d-1} ~ \cdots ~ \bm{x}_1^{p_1} ~ \cdots ~ \bm{x}_{d-1}^{p_{d-1}}]$ as in (\ref{eqn:back-gmat-def}). Define
  \begin{align}
  \bm{C} &= \bm{P}_{\bm{G}_{[-d]}^{\perp}\mathbb{X}_d^{[-0]}}+\bm{G}^{\perp}\bm{H}_{p_d,d}^*+\bm{H}_{p_d,d}^*\bm{G}^{\perp}-\bm{G}^{\perp} \bm{H}_{p_d,d}^* \bm{H}_{p_d,d}^* \bm{G}^{\perp} + O(n^{-1}h_d^{-1}\bm{I}+ n^{-1}\bm{J}),\label{eqn:cmat}\\
  \bm{D} &= \bm{G}^{\perp}- \sum_{j=1}^d \bigg\{ \bm{H}_{p_j,j}^*\bm{G}^{\perp} + O(n^{-1}h_j^{-1}\bm{I}+n^{-1}\bm{J})\bigg\}, \label{eqn:dmat}\\
  \bm{E} &= \bm{P}_{\bm{G}_{[-d]}^{\perp}\mathbb{X}_d^{[-0]}} + \bm{H}_{p_d,d}^* \bm{G}^{\perp}+O(n^{-1}h_d^{-1}\bm{I}+n^{-1}\bm{J}),\label{eqn:emat}
  \end{align}
  where $\bm{P}_{\bm{G}_{[-d]}^{\perp}\mathbb{X}_d^{[-0]}}=\bm{G}_{[-d]}^{\perp}\mathbb{X}_d^{[-0]} \left(\mathbb{X}_d^{{[-0]}^{T}} \bm{G}_{[-d]}^{\perp} \mathbb{X}_d^{[-0]} \right)^{-1}\mathbb{X}_d^{{[-0]}^{T}} \bm{G}_{[-d]}^{\perp}$, $\mathbb{X}_d^{[-0]}=[\bm{x}_d  \cdots  \bm{x}_d^{p_d}]_{n \times p_d}$, $\bm{J}$ is the matrix of ones, and $\bm{I}$ is an identity matrix of size $n$.

\subsection{The Generalized Likelihood Ratio Test}
We define the GLR test statistic analogous to  \citet{fan2005nonparametric}. Since the distribution of $\epsilon_i$ is unknown, pretend that the error distribution is normal, $\mathcal{N}(0,\sigma^2)$, to obtain the likelihood. However, we note that  normality assumption is not needed to derive asymptotic properties for GLR statistic. In Section \ref{sec:simul}, we show that asymptotic distribution of GLR statistic is robust to error distribution to some extent.  Now, the log-likelihood under model (\ref{eqn:int-model}) is
\begin{align*}
-\frac{n}{2} \log(2\pi\sigma^2) -\frac{1}{2\sigma^2}\sum_{i=1}^{n}\left(Y_i-\alpha_0-m_1(X_{i1})-\ldots-m_d(X_{id})\right)^2.
\end{align*}
Replacing $\alpha_0$, $m_k(\cdot)$, $k=1,\ldots,d$,  with their estimates $\widehat{\alpha}_0$, $\widehat{m}_k(\cdot)$,  yields
\begin{align*}
\ell(H_1) &= -\frac{n}{2} \log(2\pi\sigma^2) -\frac{1}{2\sigma^2}RSS_1,
\end{align*}
where $RSS_1=\sum_{i=1}^{n}\left(Y_i-\widehat{\alpha}_0-\widehat{m}_1(X_{i1})-\ldots-\widehat{m}_d(X_{id})\right)^2$. This likelihood  function attains maximum for $\sigma^2=\frac{1}{n}RSS_1$ which implies that $\ell(H_1)\approx-\frac{n}{2}\log(RSS_1)$. Similarly, the log-likelihood for $H_0$ is $\ell(H_0) \approx-\frac{n}{2}\log(RSS_0)$, with $RSS_0=\sum_{i=1}^n \big(Y_i-\widehat{\alpha}_0-\widetilde{m}_1(X_{i1})-$ $\ldots-\widetilde{m}_d(X_{i(d-1)})\big)^2$, and  $\widetilde{m}_k(\cdot)$, $k=1,\ldots,d-1$, are the estimators of $m_k(\cdot)$ under $H_0$, using the simplified smooth backfitting algorithm with the same set of bandwidths. Now, we define the GLR test statistic as
\begin{align}
\lambda_n(H_0) = [\ell(H_1)-\ell(H_0)] \approxeq \frac{n}{2} \log\frac{RSS_0}{RSS_1} \approx \frac{n}{2} \frac{RSS_0-RSS_1}{RSS_1},
\label{eqn:glrt}
\end{align}
and  reject the null hypothesis when $\lambda_n(H_0)$ is large. From Lemma \ref{lem:an12} (Supplementary Material), we obtain that 
\begin{align*}
  \frac{RSS_0-RSS_1}{RSS_1} &= \frac{\bm{y}^T \bm{C} \bm{y}}{\bm{y}^T \bm{D} \bm{y}}
\end{align*}
for $\bm{C}$ and $\bm{D}$ defined in (\ref{eqn:cmat}) and (\ref{eqn:dmat}), respectively. This motivates us to consider the following F-type of test \citep{huang2010analysis}
\begin{align}
  F_{\lambda} & = \frac{\bm{y}^T \bm{C} \bm{y}}{\bm{y}^T \bm{D} \bm{y}} \frac{\text{tr}(\bm{D})}{\text{tr}(\bm{C})}, \label{eqn:f-lambda}
\end{align}
where $\text{tr}(\cdot)$ denotes the trace.
Theorem \ref{thm:glr1} shows the statistic (\ref{eqn:f-lambda}) indeed follows F distribution asymptotically.

\subsection{The Loss Function Test}
\label{subsubsec:lft}
The LF testing framework \citep{hong2013loss} compares models under $H_0$ and $H_1$ via a loss function $L:\mathbb{R}^2 \rightarrow \mathbb{R}$. This is more relevant to decision-making under uncertainty in some applications. We write the discrepancy between the models as
\begin{eqnarray}
Q_n &=& \sum_{i=1}^nL\left[\widehat{m}_+(X_{i1},\ldots,X_{id}), \widetilde{m}_{+}^{(-d)}(X_{i1},\ldots,X_{i(d-1)})\right],
\label{eqn:lft-1}
\end{eqnarray}
where $\widehat{m}_+(X_{i1},\ldots,X_{id})=\widehat{\alpha}_0+\widehat{m}_1(X_{i1})+\ldots+\widehat{m}_d(X_{id})$ and  $\widetilde{m}_{+}^{(-d)}(X_{i1},\ldots,X_{i(d-1)})=\widehat{\alpha}_0+\widetilde{m}_1(X_{i1})+\ldots+\widetilde{m}_{d-1}(X_{i(d-1)})$ are the $i$th predicted values for the models under $H_1$ and $H_0$, respectively. Similar to \citet{hong2013loss}, we consider a specific class of functions called the generalized cost-of-error function defined as $L(u,v)=d(u-v)$, where $d(\cdot)$ is twice continuously differentiable with $d(0)=0$, $d'(0)=0$ and $ 0 < d''(0) <\infty$.

We define the LF test statistic as
\begin{align}
q_n(H_0) &= \frac{Q_n}{n^{-1}RSS_1}\approx \frac{d''(0)/2 \sum_{i=1}^n(\widehat{m}_+(X_{i1},\ldots,X_{id})- \widetilde{m}_{+}^{(-d)}(X_{i1},\ldots,X_{i(d-1)}))^2+R}{n^{-1}RSS_1},
\label{eqn:lft}
\end{align}
where $RSS_1$ is the residual sum of squares under alternative and $R$ is the remainder term in the Taylor expansion of $d(\cdot)$. We reject the null hypothesis  when $q_n(H_0)$ is large.

Interestingly, when the estimated additive functions  under $H_0$ and $H_1$ are approximately equal, that means, $\widehat{m}_j(X_{ij}) \approxeq \widetilde{m}_j(X_{ij})$ for $i=1,\ldots,n$ and $j=1,\ldots,d-1$, we obtain 
\begin{align*}
  Q_n  \approx \frac{d''(0)}{2} \sum_{i=1}^n \widehat{m}_d^2(X_{id}) + R.
\end{align*}
The form of the statistic $Q_n$ is similar to the backfitting based test statistic  proposed in \citet{mammen2022backfitting} as
\begin{align}
  S_n &= \int \widehat{m}_d^2(x_d) f_d(x_d) dx_d, \label{eqn:sn}
\end{align}
where $f_d(\cdot)$ is the distribution of $X_d$. For more discussion of similar tests please refer to \citet{mammen2022backfitting}.

Based on (\ref{eqn:lft}), the arguments analogous to Lemma \ref{lem:an12} (Supplementary Material) help us to define the following F-type of test statistic 
\begin{align}
  F_q &= \frac{\bm{y}^T\bm{E}^T\bm{E}\bm{y}}{\bm{y}^T\bm{D}\bm{y}} \frac{\text{tr}(\bm{D})}{\text{tr}(\bm{E}^T\bm{E})},
  \label{eqn:f-q}
\end{align}
for $\bm{E}$ defined in (\ref{eqn:emat}). This test is discussed in Theorem \ref{thm:loss1}.

%

\section{Asymptotic Results} \label{sec:asym}
  In this section, we develop asymptotic theory for the GLR  and the LF test statistics defined in Section \ref{sec:hypo-test} under model (\ref{eqn:int-model}). We derive the asymptotic null distributions of these test statistics when the testing problem is of the form (\ref{eqn:h1}) and discuss Wilks phenomenon and optimal power properties. For simplicity in theoretical arguments we assume that all the data points are interior $[2h_j, 1-2h_j]$, $j=1,\ldots,d$ \citep{huang2018}. However, we remark that, under the conditions of Theorems \ref{thm:glr1} and \ref{thm:loss1}, the  additional bias terms introduced by the boundary points are of smaller order. Therefore, our theory holds when the data include boundary points.  
  
  We list some of the assumptions required for our theoretical results in the following.

  \begin{assumption}\label{as:1-x}  The 
    densities $f_j(\cdot)$ of $X_j$  are Lipschitz-continuous and bounded away
    from 0 and have bounded support $\Omega_j$ for $j=1,\ldots,d$. The joint density of $X_j$ and $X_{j'}$, $f_{j,j'}(\cdot,\cdot)$, for $1 \le j \neq j' \le d$, is also Lipschitz continuous and have bounded support.
  \end{assumption}
  \begin{assumption} \label{as:2-ker-h}  The kernel $K(\cdot)$ is a bounded
    symmetric density function with bounded support and satisfies Lipschitz
    condition. The bandwidth $h_j \rightarrow 0$   and $n h_j^2 / (\ln n)^2 \rightarrow \infty$, $j=1,\ldots,d$, as $n \rightarrow \infty$.
  \end{assumption}
  \begin{assumption} \label{as:3-m}
    The $(2p_j+2)-$th derivative of $m_j(\cdot)$, $j=1,\ldots,d$, exists. 
  \end{assumption}
  
  \begin{assumption} \label{as:4-e}
    The error $\epsilon$ has mean 0, variance
    $\sigma^2$, and  finite fourth moment.
  \end{assumption}
  
  \begin{assumption}\label{as:5-d}
    The loss function $d: \mathbb{R} \rightarrow \mathbb{R}^+$ has a unique minimum at $0$, and
    $d(z)$ is monotonically nondecreasing as $|z| \rightarrow \infty$.
    Furthermore, $d(z)$ is twice continuously differentiable at $0$ with $d(0)=0$,
    $d'(0)=0$, $M=\frac{1}{2}d''(0) \in (0,\infty)$, and $|d''(z)-d''(0)| \le
    C|z|$ for any $z$ near $0$.
  \end{assumption}
  
  The Assumptions \ref{as:1-x}, \ref{as:2-ker-h} and \ref{as:4-e} are standard for additive models in the nonparametric smoothing literature; for example, they are similar to  \citet{fan2005nonparametric, hua14local,huang2018}. Assumption \ref{as:3-m} is required for simplified smooth backfitting to achieve bias reduction. For example, similar assumptions are found in \citet{hua14local,huang2018}. Assumption \ref{as:5-d} is from \citet{hong2013loss} and it is required for the loss function.

\subsection{Asymptotic Null Distributions of GLR and LF Tests}
  
  Let $\mu_t=\int u^t K(u) du$ and $v_t=\int u^t K^2(u) du$ for $t=0,1,\ldots$. Let $\bm{S}=(\mu_{i+j-2})$, $1 \le i,j \le (p_d+1)$ be a $(p_d+1)\times (p_d+1)$ matrix and $s^{i,j}$ are the elements of $\bm{S}^{-1}$. Denote the convolution of $K_l(x)$ with $K_m(x)$ by $K_l*K_m$, where $K_l(x)=x^lK(x)$ for $l,m=0,1,\ldots$.  
  Let,
  \begin{align*}
  \mu_n = \frac{1}{2}E\left(\sum_{i=1}^n c_{ii}\right) &=\frac{| \Omega_d |}{h_d} \bigg(\sum_{l=0}^{p_d}\sum_{m=0}^{p_d} v_{l+m} s^{(m+1),(l+1)} \\ & \qquad - \frac{1}{2} \int \bigg\{ \sum_{l=0}^{p_d}\sum_{m=0}^{p_d} (K_l*K_m)(u) (-1)^m s^{(m+1),(l+1)} \bigg\}^2 du  \bigg)+ o_p(h_d^{-1}), \\
  \sigma_{n}^2= \sum_{i<j}c_{ij}^2&= \frac{| \Omega_d |}{h_d} \int \bigg\{ \sum_{l=0}^{p_d}\sum_{m=0}^{p_d} (K_l*K_m)(u) (-1)^m s^{(m+1),(l+1)} \\&\qquad -\frac{1}{2} \int \bigg[ \sum_{l=0}^{p_d} \sum_{m=0}^{p_d} (K_l * K_m) (u+v) (-1)^m s^{(m+1),(l+1)} \bigg] \\ & \qquad \qquad \times \bigg[ \sum_{l=0}^{p_d} \sum_{m=0}^{p_d} (K_l * K_m) (v) (-1)^m s^{(m+1),(l+1)} \bigg] dv \bigg\}^2 du + o_p(h_d^{-1}), \text{ and }\\
  r_k &= \frac{2\mu_n}{\sigma_{n}^2},
  \end{align*}
  where $c_{ij}$ is the $(i, j)$th, $1 \le i, j \le n$, element of $\bm{C}$ defined in (\ref{eqn:cmat}),  $| \Omega_d |$ is the length of the support of the density $f_d(x_d)$ of $X_d$. In practice, the above asymptotic expressions are not required to compute the quantities $\mu_n$ and $\sigma_n$. We can compute them directly from the matrix $\bm{C}$ defined in (\ref{eqn:cmat}) which provides a good approximation.
  
  Hereafter, the notations ``$\xrightarrow{d}$'' and ``$\xrightarrow{p}$'' stand for convergence in distribution and probability, respectively.   The following theorem describes the Wilks type of result for the GLR test conditional on the sample space $\mathcal{X}$. 

  \begin{theorem} \label{thm:glr1} (GLR test)
    Suppose that conditions \ref{as:1-x}--\ref{as:4-e} hold and $0 \le p_j \le 3$, $j=1,\ldots,d$. Then, under
    $H_0$ for the testing problem (\ref{eqn:h1})
    \begin{align}
    P\left\{ \sigma_{n}^{-1}\left(  \lambda_n(H_0)-\mu_n-\frac{1}{2\sigma^2}d_{1n} \right) < t |\mathcal{X} \right\} \xrightarrow{d} \bm{\Phi}(t),
    \end{align}
    where  $d_{1n}=O_p\left(1+\sum_{j=1}^dnh_j^{4(p_j+1)}+\sum_{j=1}^d\sqrt{n}h_j^{2(p_j+1)}\right)$ and $\bm{\Phi}(\cdot)$ is the standard normal distribution. Furthermore, if $nh_j^{4(p_j+1)}h_d \rightarrow 0$ for $j=1,\ldots,d$, conditional on $\mathcal{X}$, 
    $ r_k\lambda_n(H_0) \xrightarrow{} \chi^2_{r_k\mu_n}$ as $n \rightarrow \infty$.
    Similarly,
    \begin{align}
    F_{\lambda} &= \frac{2\lambda_n(H_0) \text{tr}(\bm{D})}{n~ \text{tr}(\bm{C})} \xrightarrow{} F_{\text{tr}(C),\text{tr}(D)},
    \label{eqn:th1-glr-f}
    \end{align}
    as $n \rightarrow \infty$, where $\text{tr}(C)$ and $\text{tr}(D)$ are the corresponding degrees of freedom.
  \end{theorem}
  Theorem \ref{thm:glr1} gives the asymptotic null distribution of the GLR test statistic for the testing problem (\ref{eqn:h1}) under $H_0$. In our opinion, the asymptotic expression for $d_{1n}$ is complicated and might not be necessary. 
  
  \begin{remark}
    The factors $r_k$ and $\mu_n$ in Theorem \ref{thm:glr1} do not depend on the nuisance parameters and nuisance functions. Therefore, the GLR test statistic $\lambda_n$ achieves the Wilks phenomenon that its asymptotic distribution does not depend on nuisance parameters and nuisance functions. Theorem \ref{thm:glr1} is different from Theorem 1 of \citet{fan2005nonparametric} because it uses simplified smooth backfitting instead of backfitting for estimation of additive components.
  \end{remark}
  
  \begin{remark}
    Theorem \ref{thm:glr1} shows that the bias  $d_{1n}$ is negligible under the condition $C_1: nh_j^{4(p_j+1)}h_d \rightarrow 0$ which is different from the condition $C_2: nh_j^{2(p_j+1)}h_d \rightarrow 0$ in Theorem 1 of \citet{fan2005nonparametric}. Suppose $h_j^{p_j+1}=O(h_d^{p_d+1})$, then the proposed GLR test achieves Wilks phenomenon for the bandwidths of order $h_j\sim n^{-1/(2p_j+3)}$ for odd $p_j$, which are the optimal bandwidths used for estimation  in  \citet{fan2005nonparametric}, while their GLR test statistic  does not. To see this, consider $p_j=1$, then $h_{opt} \sim n^{-1/5}$, the first condition $C_1: nh^{9} \sim n n^{-9/5} =n^{-4/5} =o(1)$ holds where as the second condition $C_2: nh^5 \sim nn^{-5/5}=O(1)$ does not hold. 
  \end{remark}
  


  We now derive the asymptotic null distribution of the LF test statistic. Let $\bm{e}_k=(e_{1k},\ldots,e_{nk})^T$ where $e_{ij}$ is the $(i,j)$th, $1 \le i,j \le n$, element of $\bm{E}$ defined in (\ref{eqn:emat}).
  Define
  \begin{align*}
  \nu_n = E(\text{tr}(\bm{E}^T\bm{E})) &=E(\sum_{i=1}^n \bm{e}_{i}^T \bm{e}_i) \\ &=  \frac{| \Omega_d |}{h_d} \int \bigg\{ \sum_{l=0}^{p_d}\sum_{m=0}^{p_d} (K_l*K_m)(u) (-1)^m s^{(m+1),(l+1)} \bigg\}^2 du + o_p(h_d^{-1}), \\
  \delta_{n}^2 = \sum_{j \neq j'}^n (\bm{e}_j^T\bm{e}_{j'})^2 &= \frac{| \Omega_d |}{h_d} \int \bigg\{ \int \bigg[ \sum_{l=0}^{p_d} \sum_{m=0}^{p_d} (K_l * K_m) (u+v) (-1)^m s^{(m+1),(l+1)} \bigg] \\ & \qquad \times \bigg[ \sum_{l=0}^{p_d} \sum_{m=0}^{p_d} (K_l * K_m) (v) (-1)^m s^{(m+1),(l+1)} \bigg] dv \bigg\}^2 du + o_p(h_d^{-1}),\\
  \text{ and } s_k &= \frac{2\nu_n}{\delta_{n}^2}.
  \end{align*}
  Denote $M=d''(0)/2$ where $d(\cdot)$ is the loss function given in Section \ref{subsubsec:lft}.
  The following theorem describes the Wilks type of result for the LF test statistic conditional on the sample space $\mathcal{X}$. 

  \begin{theorem}\label{thm:loss1} (LF test)
    Suppose that conditions \ref{as:1-x}--\ref{as:5-d} hold and $0 \le p_j \le 3$, $j=1,\ldots,d$. Then, under
    $H_0$ for the testing problem (\ref{eqn:h1})
    \begin{align}
    P\left\{ \delta_n^{-1} \left(  \frac{q_n(H_0)}{M}-\nu_n-\frac{1}{\sigma^2}b_{1n} \right) < t |\mathcal{X} \right\} \xrightarrow{d} \bm{\Phi}(t),
    \end{align}
    where  $b_{1n}=O_p\left(1+\sum_{j=1}^d nh_j^{4(p_j+1)}\right)$. Furthermore, if $nh_j^{4(p_j+1)}h_d \rightarrow 0$ for $j=1,\ldots,d$, conditional on $\mathcal{X}$, 
    $ s_k M^{-1}q_n(H_0) \xrightarrow{} \chi^2_{s_k\nu_n}$ as $n \rightarrow \infty$.
    Similarly,
    \begin{align}
    F_q &=  \frac{q_n(H_0)}{Mn} \frac{\text{tr}(\bm{D})}{\text{tr}(\bm{E}^T\bm{E})} \xrightarrow{} F_{\text{tr}(\bm{E}^T\bm{E}),\text{tr}(\bm{D})},
    \end{align}
    as $n \rightarrow \infty$.	
  \end{theorem}
  \begin{remark}
    Theorem \ref{thm:loss1} shows that the factors $s_k$ and $v_n$ do not depend on the nuisance parameters and nuisance functions. Therefore,  like the GLR statistic, the LF test statistic also enjoys Wilks phenomenon that its asymptotic distribution does not depend on nuisance parameters and nuisance function. Further, since Wilks phenomenon is achieved for optimal bandwidths $h_j\sim n^{-1/(2p_j+3)}$, for odd $p_j$, $j=1,\ldots,d$, undersmoothing may not be necessary. 
  \end{remark} 

  \begin{remark}
   The LF test statistic includes an extra scaling constant $M$, which is the curvature of the loss function. When $M$ is correctly specified, the asymptotic distribution of the scaled statistic does not depend on it. However, the choice of $M$ is irrelevant if the conditional bootstrap method (Supplementary Material) is used to simulate the null distribution. We further validate this using simulations in Section \ref{sec:simul}. We also provide more discussion on the efficiency of loss functions in Theorem \ref{thm:glr-lft-h1}.
  \end{remark}
  

  Unlike GLR test statistic, the LF test statistic includes only the second-order term in the Taylor expansion because the first-order term vanishes to 0 under $H_0$. For univariate nonparametric model (\ref{eqn:int-model}) with $d=1$, \citet{hong2013loss} argue that not having first-order term could be one reason for LF test to be asymptotically more powerful than GLR test. In Theorem \ref{thm:are} we show that similar result holds for the nonparametric additive model (\ref{eqn:int-model}). We now discuss the optimal power properties of the proposed test statistics in the following section.


  \subsection{Power of GLR and LF Tests}
  
  We consider the framework of \citet{fan2001generalized} and \citet{fan2005nonparametric} to study the power of GLR and LF tests. Assume that $h_d=o(n^{-1/(4p_d+5)})$, so that the second term in both $d_{1n}$ and $b_{1n}$ is of smaller order than $\sigma_n$ and $\delta_n$, respectively. We note that the optimal bandwidth for the testing problem (\ref{eqn:h1}) is $h_d=O(n^{-2/(8p_d+9)})$ (to be shown later in Theorem \ref{thm:opt-test}), which satisfies the condition $h_d=o(n^{-1/(4p_d+5)})$. 
  Under these assumptions, Theorems \ref{thm:glr1} and \ref{thm:loss1} lead to
  the following  approximate level $\alpha$ tests for GLR and  LF test statistics, respectively
  \begin{align*}
    \phi_{\lambda_n} = I\{\lambda_n(H_0)-\mu_n \ge z_{\alpha}\sigma_{n}\},\qquad
    \phi_{q_n} = I\{\frac{q_n(H_0)}{M}-\nu_n \ge z_{\alpha}\delta_{n}\}.
  \end{align*}
  Let $\mathcal{M}_n$ be a class of functions such that any $M_n \in \mathcal{M}_n$ satisfy the following regularity conditions as stated in \citet{fan2005nonparametric}:
  \begin{align}
    \begin{split}
    \text{var}(M_n^2(X_d))  \le K(E[M_n^2(X_d)])^2, \qquad
    nE[M_n^2(X_d)] > K_n \rightarrow \infty,
    \end{split}
    \end{align} 
    for some constants $K > 0$ and $K_n \rightarrow \infty$.
  Let $\eta= 2(p_d+1)$ with $0 \le p_d \le 3$. Define a class of functions,
  \begin{align}
    \mathcal{M}_n(\rho; \eta) =\big\{M_n \in \mathcal{M}_n: &E[M_n^2(X_d)]\ge \rho^2,  E[\nabla^{r}M_n(X_d)]^2 \le R_*^2 \text{ with } r \le \eta \big\},
    \label{eqn:class-h1}
  \end{align}
  for a given $\rho > 0$, where $\nabla^{r}M_n(X_d)$ is the $r$th derivative of $M_n$ and $R_*$ is some positive constant.  
  Consider the contiguous alternative of the form
  \begin{align}
    H_{1n}:m_d(X_d) &= M_n(X_d),
    \label{eqn:h1a}
  \end{align}
  where $M_n(X_d) \rightarrow 0$ and $M_n \in \mathcal{M}_n(\rho; \eta)$.
  
  The following theorem is useful to approximate the power of  GLR and LF tests under the contiguous alternative (\ref{eqn:h1a}).

  \begin{theorem} \label{thm:glr-lft-h1}
    Suppose $E\{M_n(X_d)|X_1,\ldots,X_{d-1}\}=0$ and $h_d\cdot\sum_{i=1}^n M_n^2(X_{id}) \xrightarrow{P} C_M$ for some constant $C_M$. Suppose $0 \le p_j \le 3$, for $j=1,\ldots,d$.
    \begin{enumerate}[label=(\roman*)]
      \item{[GLR test]} Suppose that conditions \ref{as:1-x}-\ref{as:4-e} hold. Under $H_{1n}$ for the testing problem (\ref{eqn:h1}),
      \begin{align*}
      P\left\{\sigma_{n}^{-1}\left(\lambda_n(H_0)-\mu_n-\frac{d_{1n}+d_{2n}}{2\sigma^2}\right)<t|\mathcal{X}\right\} \xrightarrow{d} \bm{\Phi}(t),
      \end{align*}
      where $\mu_n$, $d_{1n}$ and $\sigma_n$ are same as those in Theorem \ref{thm:glr1} and
      \begin{align*}
      d_{2n} &= \sum_{i=1}^n M_n^2(X_{id}) (1+o_p(1)).
      \end{align*}
      \item{[LF test]} Suppose that conditions \ref{as:1-x}-\ref{as:5-d} hold. Under $H_{1n}$ for the testing problem (\ref{eqn:h1}),
      \begin{align*}
      P\left\{\delta_{n}^{-1}\left(\frac{q_n(H_0)}{M}-\nu_n-\frac{b_{1n}+b_{2n}}{\sigma^2}\right)<t|\mathcal{X}\right\} \xrightarrow{d} \bm{\Phi}(t),
      \end{align*}
      where $\nu_n$, $b_{1n}$ and $\delta_{n}$ are same as those in Theorem \ref{thm:loss1} and 
      \begin{align*}
      b_{2n} &= \sum_{i=1}^n M_n^2(X_{id}) (1+o_p(1)). 
      \end{align*} 
    \end{enumerate}
    
  \end{theorem}
  
  Theorem \ref{thm:glr-lft-h1} shows that when $nh_j^{4(p_j+1)}h_d \rightarrow 0$, $j=1,\ldots,d$, the alternative distributions are independent of the nuisance functions $m_j(x_j)$, $j \neq d$, and this helps us to compute the power of the tests via simulations over a large range of bandwidths with nuisance functions fixed at their estimated values.
  
  It is interesting to note that the noncentrality parameters in part (ii) of Theorem \ref{thm:glr-lft-h1} are independent of the curvature parameter $M=d''(0)/2$ of the loss function $d(\cdot)$. This implies that, as discussed in \citet{hong2013loss}, all loss functions satisfying Assumption \ref{as:5-d} are asymptotically equally efficient under $H_1$ in terms of Pitman's efficiency criterion [\citet{pitman2018some}, Chapter 7].
  
  The maximum of the probabilities of  type II errors is
  \begin{align}
    \beta_{\lambda_n}(\alpha,\rho) = \underset{M_n \in \mathcal{M}_n(\rho;\eta)}{\sup} \beta_{\lambda_n}(\alpha,M_n), \qquad \beta_{q_n}(\alpha,\rho) = \underset{M_n \in \mathcal{M}_n(\rho; \eta)}{\sup} \beta_{q_n}(\alpha,M_n),
  \end{align}
  where $\beta_{\lambda_n}(\alpha,M_n)=P(\phi_{\lambda_n}=0|m_d=M_n)$  and $\beta_{q_n}(\alpha,M_n)=P(\phi_{q_n}=0|m_d=M_n)$ are the probabilities of  type II
  errors at the alternative $H_{1n}:m_d=M_n$. Use $\beta(\alpha,\rho)$ to denote either $\beta_{\lambda_n}(\alpha,\rho)$ or $\beta_{q_n}(\alpha,\rho)$ and $\phi$ to denote $\phi_{\lambda_n}$ or $\phi_{q_n}$. As mentioned in \citet{fan2001generalized} and \citet{fan2005nonparametric}, the minimax rate of $\phi_{\lambda_n}$ or $\phi_{q_n}$ is defined as the smallest $\rho_n$ such that:
  \begin{enumerate}[label=(\alph*)]
    \item for every $\rho > \rho_n$, $\alpha > 0$, and for any $\beta > 0$, there exists a constant $c$ such that $\beta(\alpha,c\rho) \le \beta+o(1)$, and
    \item for any sequence $\rho_n^*=o(\rho_n)$, there exists $\alpha>0$ and $\beta>0$ such that for any $c>0$, $P(\phi=1|m_d=M_n)=\alpha + o(1)$ and $\lim \inf_n \beta(\alpha, c\rho_n^*) > \beta$.
  \end{enumerate}
  
  The following theorem provides the rate with which  the alternatives can be detected by GLR $(\phi_{\lambda_n})$ and LF  $(\phi_{q_n})$ tests. The convergence rate depends on bandwidth.

  \begin{theopargself}
  \begin{theorem}\label{thm:opt-test}
    Under conditions \ref{as:1-x}-\ref{as:5-d}, if
    $h_k^{2(p_k+1)}=O(h_d^{2(p_d+1)})$ and $0 \le p_k \le 3$, for $k=1,\ldots,d-1$, then for the testing
    problem (\ref{eqn:h1}), both GLR and LF tests can detect alternatives
    with rate  $\rho_n=n^{-\frac{4(p_d+1)}{8p_d+9}}$ when $h_d=c_*n^{-\frac{2}{8p_d+9}}$ for some
    constant $c_*$. 
  \end{theorem}
\end{theopargself}
  \begin{remark}
    For the class of alternatives $\mathcal{M}_n(\rho;\eta)$ in (\ref{eqn:class-h1}), the rate of convergence for nonparametric hypothesis testing  according to the formulations of \citet{ingster1993asymptotically} and \citet{spokoiny1996adaptive} is $n^{-\frac{2\eta}{4\eta+1}}$ where $\eta$ is the smoothness parameter. Since $\eta=2(p_d+1)$ in this study, the GLR and LF tests are asymptotically optimal based on their rates given in Theorem \ref{thm:opt-test}. Our rates are different from the rates in Theorem 5 of  \citet{fan2005nonparametric} because of different smoothness parameter $\eta=(p_d+1)$ considered in their study. For this reason, the optimal bandwidth for testing in our study $n^{-\frac{2}{8p_d+9}}$ which is also different from  $n^{-\frac{2}{4p_d+5}}$ in \citet{fan2005nonparametric}. 
  \end{remark}
  
  \begin{remark}
    Based on  Theorems \ref{thm:glr1} and \ref{thm:loss1}, the assumption on bandwidths $nh_j^{4(p_j+1)}h_d=o(1)$, $j=1,\ldots,d$, is required to ensure Wilks property for both GLR and LF tests. This is true for a collection of bandwidths $h_j \in (0,n^{-\frac{1}{4p_j+4}}]$ which includes the optimal bandwidths  $n^{-\frac{1}{2p_j+3}}$ used in backfitting \citep{fan2005nonparametric}. With our method, the bandwidths well suited for curve estimation might also be useful for testing. 
  \end{remark}
  We now show that the LF test is asymptotically more powerful than the GLR test. For ease of exposition, we assume that $p_j=0$, $j=1,\ldots,d$. Without loss of generality, let $M_n(x_d)=n^{-1/2}h_d^{-1/2}g(x_d)$ which satisfy the condition in Theorem \ref{thm:glr-lft-h1}. We now compare the relative efficiency between the LF test statistic $q_n$ and the GLR test statistic $\lambda_n$ under the class of local alternatives
  \begin{align}
    H_n: m_d(x_d) = n^{-1/2}h_d^{-1/2}g(x_d),
    \label{eqn:h1-spl}
  \end{align} 
  where $E(g(X_d)|X_1,\ldots,X_{d-1})=0$ and $\sum_{i=1}^n g^2(X_{id})=O_p(h_d^{-1})$.
  While Theorem \ref{thm:opt-test} shows that the GLR  and the LF tests achieve optimal rate of convergence in the sense of \citet{ingster1993asymptotically} and \citet{spokoiny1996adaptive}, Theorem \ref{thm:are} provides that under the same set of regularity conditions, the LF test is asymptotically more powerful than the GLR test under $H_n$ in (\ref{eqn:h1-spl}). 

  \begin{theorem}
    \label{thm:are}[Relative efficiency] Suppose Conditions \ref{as:1-x}--\ref{as:5-d} hold, $h \propto n^{-\omega}$ for $\omega \in (0,1/(4p_d+5))$ and $p_j = 0$ for $j=1,\ldots,d$. Then Pitman's relative efficiency of the LF test over the GLR test under  $H_n$ in (\ref{eqn:h1-spl}) is given by
    \begin{align*}
      \text{ARE}&(q_n,\lambda_n)  \\
      &=  \left[ \frac{\int \left\{2(K_0*K_0)(u) -\int (K_0*K_0)(u+v) (K_0*K_0)(v) dv \right\}^2 du}{\int \left\{\int (K_0*K_0)(u+v) (K_0*K_0)(v) dv \right\}^2 du} \right]^{1/(2-3\omega)}.
      \end{align*}
      The asymptotic relative efficiency ARE$(q_n,\lambda_n)$ is larger than 1 for any kernel satisfying Condition \ref{as:2-ker-h} and $K(\cdot) \le 1$.
  \end{theorem}

\begin{remark}
Theorem \ref{thm:are} shows that the Pitman's relative efficiency of the LF test over the GLR test is larger than 1 for $p_j=0$, $j=1,\ldots,d$, which means that the LF test is asymptotically more efficient than the GLR test. Given the complicated expressions of $\sigma_n$ and $\delta_n$, the extension of Theorem \ref{thm:are}  to general $p_j=1,2,3$, is not straightforward. We defer this for future research.  
\end{remark}

  \begin{remark}
  The result in Theorem \ref{thm:are} does not imply that the GLR test is not useful. The GLR test is a natural extension of classical Likelihood Ratio test with many desirable features and has been widely used in the literature. As stated in \citet{hong2013loss}, same bandwidths and same kernel functions $K(\cdot)$ are required for the relative efficiency of $q_n$ over $\lambda_n$ to hold. Therefore, it might be possible that both test statistics achieve similar efficiencies under different bandwidths and kernel functions. In our simulations in Section \ref{sec:simul}, we observe that the statistic $q_n$  achieves larger powers than the statistic $\lambda_n$.   
\end{remark}

  \begin{remark}
    The result in Theorem \ref{thm:are} is new to the literature. While Theorem 4 in \citet{hong2013loss} discuss the asymptotic relative efficiency of $q_n$ over $\lambda_n$  for  Nadaraya-Watson estimator in an univariate model, the proposed Theorem \ref{thm:are} discuss the relative efficiency for similar type of estimators using $\bm{H}_{0,j}^{*}$, $j=1,\ldots,d$, in additive models.       
  \end{remark}

\section{Numerical Comparison of GLR and LF Tests}\label{sec:numeric}
In this section, we evaluate the performance of GLR and LF tests in finite samples. Using simulations, we demonstrate the Wilks phenomenon and examine the effect of error distribution on the performances of  GLR and  LF tests. Local linear smoothing with Gaussian kernel is considered in all the simulations. We use software \texttt{Julia} \citep{bezanson2017julia} to carry out simulations and data analysis. We also illustrate the usefulness of the proposed statistics using Boston housing data. Additional results are presented in Section S3.2 of supplementary Material. 

\subsection{Simulations} \label{sec:simul}

We mimic the simulation designs in \citet{fan2005nonparametric} and \citet{huang2018}. Consider the additive model,
\begin{align}
  Y= m_1(X_1) + m_2(X_2) + m_3(X_3) + m_4(X_4) + \epsilon,
  \label{eq:sim-model}
\end{align}
where $m_1(X_1)= 0.5-X_1^2+3X_1^3$, $m_2(X_2)=\sin(\pi X_2)$, $m_3(X_3)=X_3(1-X_3)$, $m_4(X_4)=\exp(2X_4-1)$, and  $\epsilon$ is distributed as $\mathcal{N}(0,1)$. For the covariates $X_1$, $X_2$, $X_3$ and $X_4$, we first simulate normally distributed random variables $Z_1$, $Z_2$, $Z_3$ and $Z_4$ with mean $[0, 0,0,0]$ and covariance $0.4\bm{I}_4+0.6\bm{1}\bm{1}^T$,  and project them back on to $[-1, 1]$ using the transformation $X_i= 2\tan^{-1}(Z_i)/\pi$, $i=1,2,3,4$. 

We consider the null hypothesis $H_0: m_2(x_2)=0$ and treat $m_1(x_1)$, $m_3(x_3)$ and $m_4(x_4)$ as nuisance functions. For comparison purpose, we implement the GLR test in \citet{fan2005nonparametric} which is based on classical backfitting; hereafter, denoted as GLR(FJ) and the corresponding test statistic as $\lambda_n(FJ)$. We also implement the backfitting based test (\ref{eqn:sn}) proposed in \citet{mammen2022backfitting}; hereafter, denoted as SB. The SB statistic in (\ref{eqn:sn}) is approximated using Riemann sum as 
\begin{align*}
  S_n &= \sum_{i=1}^n \widehat{m}^2_2(X_{(i)2}) \widehat{f}_2(X_{(i)2}) (X_{(i)2}-X_{(i-1)2}), ~ \text{  with  }  ~ X_{(0)2}:=X_{(1)2},
\end{align*}
where $X_{(i)2}$ is the $i$th order statistic of $X_2$.
An \texttt{R} \citep{team2013r} package \emph{wsbackfit} \citep{wsbackfit} is used to estimate the  additive component $\widehat{m}_2$ and kernel density estimate  $\widehat{f}_2$  which uses the bandwidth considered for $\widehat{m}_2(\cdot)$. 

 We compute the optimal bandwidths using the following cross-validation procedure which is defined similar to \citet{nielsen2005smooth}.
 \begin{enumerate}
  \item Fit the additive model for initial values of bandwidths. 
  \item For any direction (covariate), consider the corresponding partial residual as a response variable and use an Akaike Information Criterion based smoothing parameter selection method \citep{hurvich1998smoothing} to determine the optimal bandwidth in univariate local linear regression.
  \item  Refit the model with the updated bandwidths and proceed to the step 2 choosing a different direction.  
  \item Obtain, the optimal bandwidths at the convergence of the above procedure. 
 \end{enumerate}

   To demonstrate Wilks phenomenon for  GLR, LF, and F tests, we choose three levels of bandwidths $h_1= h_{1,\text{opt}}/3, h_{1, \text{opt}}, 1.5 h_{1,\text{opt}}$ and $h_2={h_{2,\text{opt}}}$, $h_3={h_{3,\text{opt}}}$ and $h_4={h_{4,\text{opt}}}$. Similarly,  we consider three levels of $m_1(X_1)$ to show that the proposed tests do not depend on the nuisance function $m_1(X_1)$:
\begin{align*}
   m_{1,\beta} (X_1) &= \left[ 1 + \beta \sqrt{\text{var}(0.5-X_1^2+3X_1^3)} \right](0.5-X_1^2+3X_1^3),
\end{align*}
where $\beta= -1.5, 0, 1.5$. For LF statistic, we consider the following class of LINEX functions \citep{hong2013loss}:
\begin{align}
  d(z) &= \frac{t}{s^2} \left[ \exp(sz)-(1+sz) \right],
\label{eqn:linex-def}
\end{align}
 where $d(z)$ is an asymmetric loss function for each pair of parameters $(s,t)$. The magnitude of $s$, which is a shape parameter, controls the degree of asymmetry. The parameter $t$ is a scale factor.

We draw 1000 samples of 100 observations from (\ref{eq:sim-model}) and for each sample, we compute the scaled GLR, LF, and F test statistics. The distributions of scaled GLR, LF, and F test statistics among 1000 simulations are obtained via a kernel estimate using the rule of thumb bandwidth $h= 1.06sn^{-.2}$, where $s$ is the standard deviation of the test statistics. Figure \ref{fig:sim1-1a} shows the estimated densities for the scaled  GLR and LF, and F test statistics. Plots in the top row show that the null distributions of scaled GLR and LF statistics follow a  chi-squared distribution over a wide range of bandwidth values for $h_1$. It is interesting to note that $r_k \approx s_k$, $\mu_n \approx v_n$ and $M=1/2$. Due to the extra scaling constant $M=1/2$, the distribution of $s_kM^{-1}q_n$ seems like a scaled version of the distribution of $r_k\lambda_n$. Both F statistics, $F_{\lambda}$ and $F_q$, are computed using the $\bm{C}$, $\bm{D}$ and $\bm{E}$ matrices defined in (\ref{eqn:cmat},\ref{eqn:dmat},\ref{eqn:emat}); the results also illustrate that they  provide a good approximation.  Similarly, plots from the bottom row demonstrate the Wilks phenomenon for scaled GLR, LF, and F test statistics, as their null distributions are nearly the same for three different choices of the nuisance functions for $m_1(\cdot)$. For LF test, we consider the LINEX loss function (\ref{eqn:linex-def}) with the choice $s=0$ and $t=1$. 

\begin{figure}[t]
  \centering
  \makebox{\includegraphics[scale=0.4]{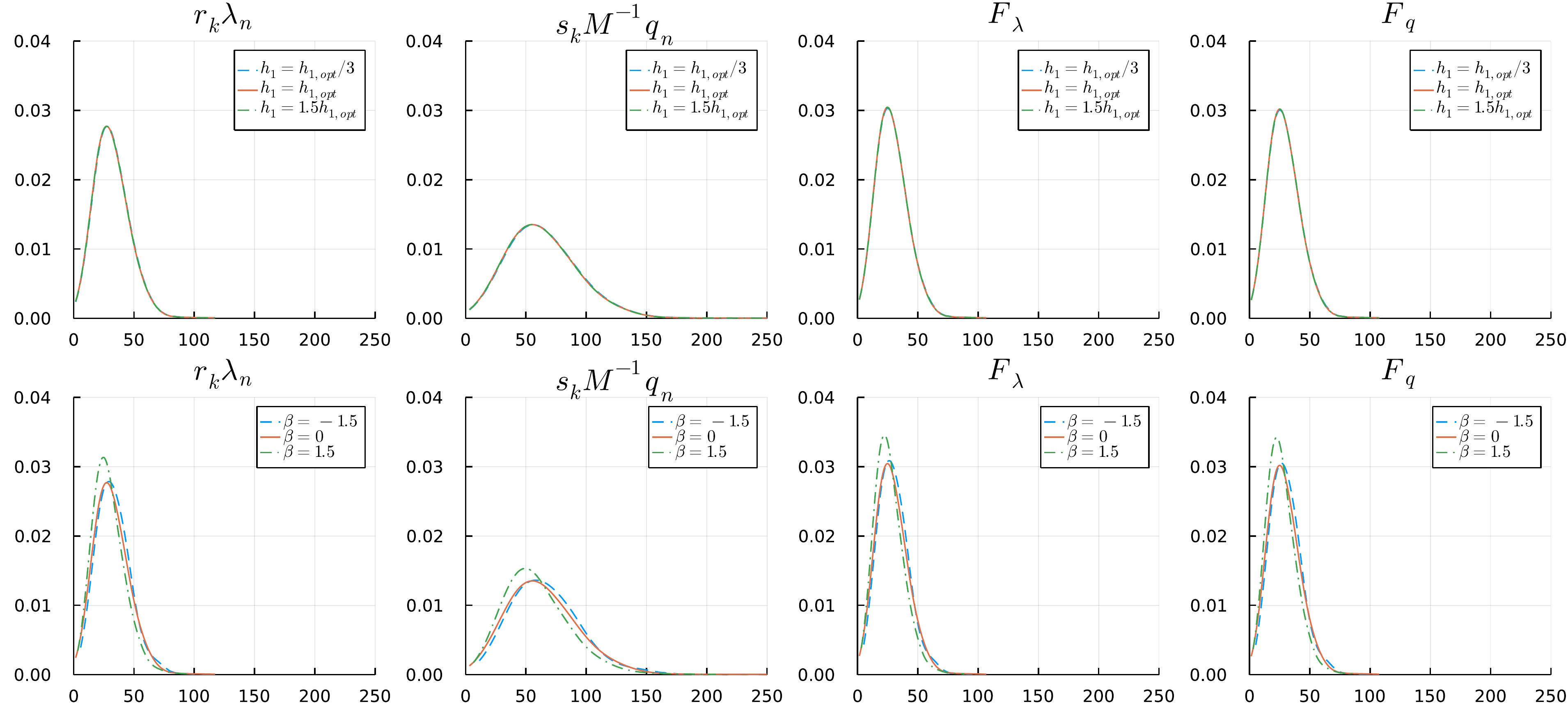}}
   \caption{Estimated densities for the scaled GLR, LF,  and F test statistics among 1000 simulations. (Top row) With fixed $(h_2,h_3,h_4)$ $=\{h_{2,\text{opt}},h_{3,\text{opt}},h_{4,\text{opt}}\}$, but different bandwidths for $h_1$ (- - $h_1=h_{1,\text{opt}}/3$; ---$ h_1=h_{1,\text{opt}}$; $-\cdot- h_1=1.5h_{1,\text{opt}}$). (Bottom row) With different nuisance functions and optimal bandwidths $h_j=h_{j,\text{opt}}$, $j=1,2,3,4$, (- - $\beta=-1.5$; ---$\beta=0$; $- \cdot- \beta=1.5$). The LF test considers the class of LINEX functions (\ref{eqn:linex-def})with $s=0$ and $t=1$. }
   \label{fig:sim1-1a}
\end{figure}

For power comparison among GLR, LF, F, GLR(FJ), and SB tests, we evaluate the power for a sequence of alternative models indexed by $\theta$,
\begin{align}
  H_{\theta}: m_{2,\theta} = \theta \sin(\pi X_2), \qquad 0 \le \theta \le 1,
  \label{eqn:htheta}
\end{align}
where $E(\sin(\pi X_2)|X_1,X_3,X_4)=0$, $\theta=0$ gives the null model and $\theta>0$ makes the alternative reasonably far away from the null model. For each given value of $\theta$, we consider $2000$ Monte Carlo replicates for calculation of the critical values via conditional bootstrap method which is described in Section S3 of Supplementary Material.
The rejection percentage values are computed based on $500$ simulations. When $\theta=0$, the alternative is identical to the null hypothesis and the power is approximately equal to the significance level $\alpha= 0.05$ or $0.01$.
Furthermore, to illustrate the influence of different error distributions on the power of GLR, LF, and F tests, we consider model (\ref{eq:sim-model}) with different error distributions of $\epsilon$, namely, $\mathcal{N}(0,1)$, $t(5)$, $\chi^2(5)$ and $\chi^2(10)$. The distributions of test statistics among 1000 simulations are provided in Figure \ref{fig:sim1-1b} in Supplementary Material. The estimated densities are approximately similar across different error distributions.   

The power of GLR, LF, F, GLR(FJ), and SB tests for the alternative model sequence in (\ref{eqn:htheta}) at the significance level $\alpha=0.05$ is provided in Figure \ref{fig:sim-five-bseven}  for  $n=100$ and $n=400$.  Figure \ref{fig:sim-five-bseven} illustrates that both GLRT and LFT differentiate the null and alternative hypotheses with high power while not being sensitive to error distributions. When $\theta=0$, the alternative is identical to the null and hence, the power should be approximately equal to $\alpha$ ($0.05$); this is evident from the results. This gives an indication that Monte Carlo approach yields the correct estimator of the null distribution. Based on Theorem \ref{thm:opt-test}, we consider the bandwidths that are optimal for testing $h_j=S_{X_j}n^{-2/17}$, $j=1,2,3,4$. Here $S_{X_j}$ is the standard deviation of $X_j$. One important observation is that the results for the statistic $F_{\lambda}$ exhibit some variation. However, in Figure \ref{fig:sim-five-bopt} where the optimal bandwidths for model estimation (cross-validation) are considered, we observe that $F_{\lambda}$ performs very similar to other statistics. We note that the optimal bandwidths are larger than $S_{X_j}n^{-2/17}$; it seems $F_{\lambda}$ is not stable for smaller bandwidths due to approximation. Overall, Figures \ref{fig:sim-five-bseven} and \ref{fig:sim-five-bopt} illustrate that the proposed methods in the study work well with the finite samples and comparable to other existing methods in the literature.


\begin{figure}
  \centering
  \includegraphics[scale=0.35]{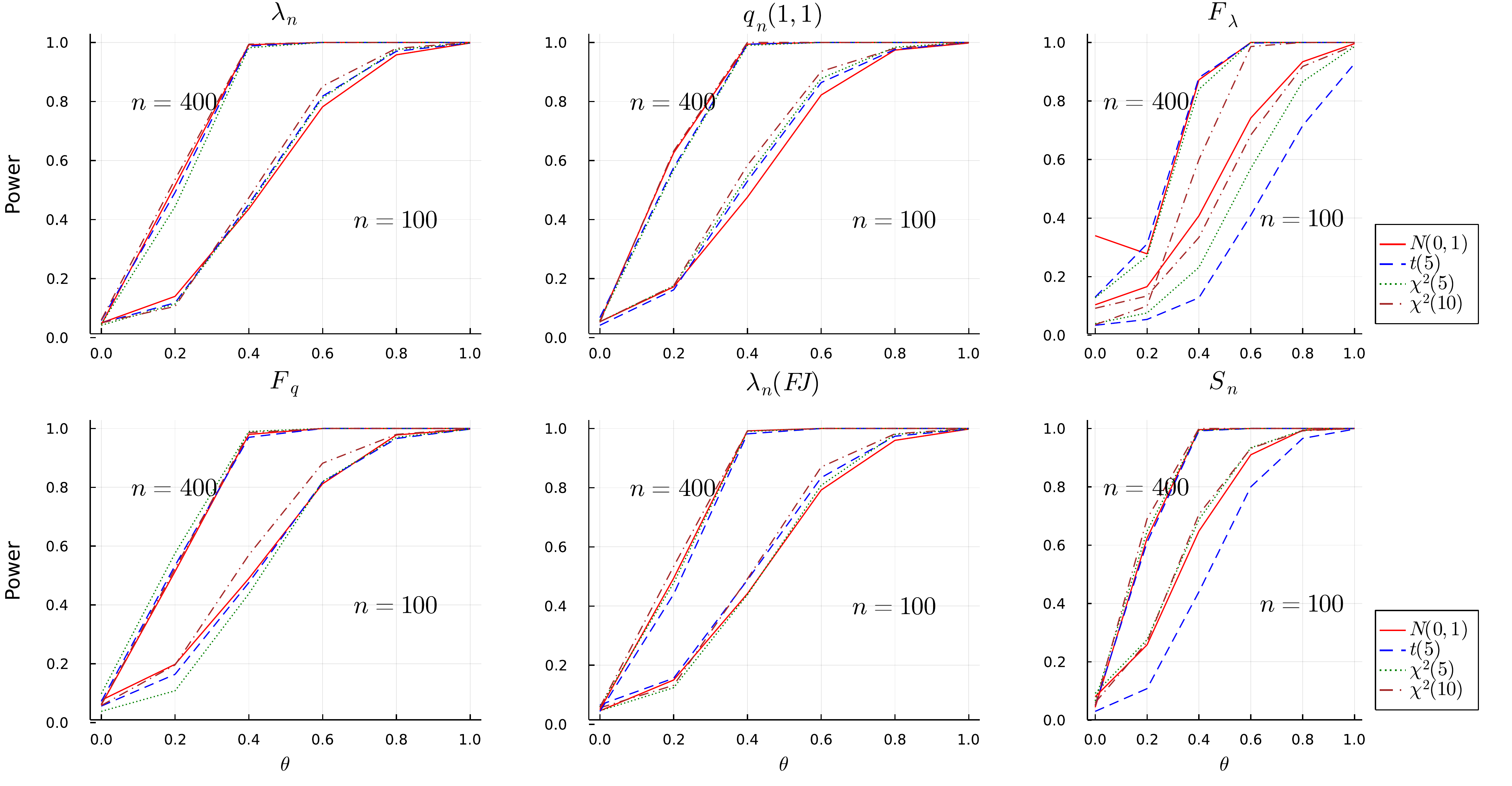}
  \caption{Power of the tests under alternative model sequence (\ref{eqn:htheta}) using optimal bandwidths for testing, $S_X n^{-2/17}$, at 5\% level of significance. Only the LF test with LINEX loss function (\ref{eqn:linex-def}) for $s=1,t=1$ is reported. The power values are similar for other choices of $s$ and $t$. }
  \label{fig:sim-five-bseven}
\end{figure}
We also provide the power comparison of the above methods at 1\% level of significance. The results are available in Figures \ref{fig:sim-one-bseven} and \ref{fig:sim-one-bopt} in Supplementary Material. The findings remain similar. 

\begin{figure}
  \centering
  \includegraphics[scale=0.35]{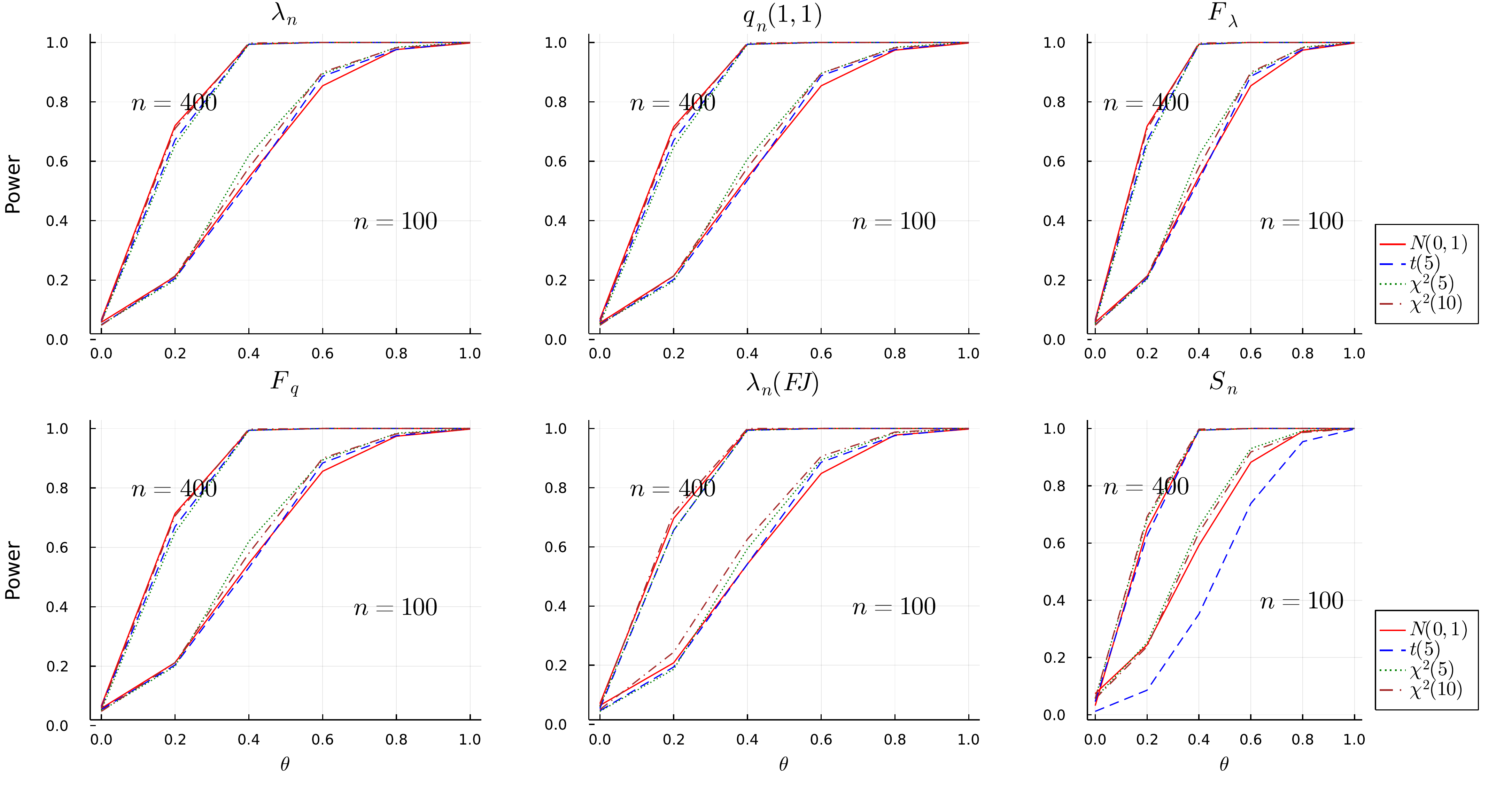}
  \caption{Power of the tests under alternative model sequence (\ref{eqn:htheta}) using optimal bandwidths for estimation (cross-validation) at 5\% level of significance. Only the LF test with LINEX loss function (\ref{eqn:linex-def}) for $s=1,t=1$ is reported. The power values are similar for other choices of $s$ and $t$. }
  \label{fig:sim-five-bopt}
\end{figure}

\subsection{Boston Housing Data Analysis}
To demonstrate the usefulness of the proposed GLR and LF tests, we consider the Boston housing data.  This data include the information collected by the U.S Census Service regarding housing in the area of Boston Mass, and originally published in \citet{harrison1978hedonic}. It contain the median values of 506 homes along with 13 sociodemographic and related variables. This data has been previously used in the literature to benchmark different algorithms and to illustrate different methodologies. For example, please see \citet{belsley2005regression}, \citet{breiman1985estimating}, \citet{opsomer1998fully} and \citet{fan2005nonparametric}. For the sake of easy comparison,  we consider the same dependent and independent variables used in \citet{fan2005nonparametric}.

\begin{itemize}
  \item \emph{MV}, median value of owner-occupied homes in \$1,000's
  \item \emph{RM}, average number of rooms per dwelling
  \item \emph{TAX}, full-value property tax rate $(\$/\$10,000)$
  \item \emph{PTRATIO}, pupil/teacher ratio by town school district
  \item \emph{LSTAT}, proportion of population that is of ``lower status" $(\%)$.
\end{itemize}

\citet{opsomer1998fully} and \citet{fan2005nonparametric} analyze this data by considering the following four-dimensional additive model,
\begin{align}
  E[MV \mid X_1, X_2, X_3, X_4] &= m_0 + m_1(X_1) +m_2(X_2) +m_3(X_3) + m_4(X_4),
  \label{eqn:bd-alt}
\end{align}
where $X_1=RM$, $X_2=\log(TAX)$, $X_3=PTRATIO$, and $X_4=\log(LSTAT)$. We use simplified smooth backfitting algorithm in Section \ref{subsubsec:simple-sb} with local linear smoothing to estimate model (\ref{eqn:bd-alt}) after six outliers $(\widehat{\epsilon}_i < -11 \text{ or } \widehat{\epsilon}_i > 12)$ are removed.  To alleviate the effect of sample size on $p$-value, we take a random sample $n=200$ observations for hypothesis testing. The optimal bandwidths  are selected using the cross-validation procedure described in Section \ref{sec:simul} which uses the AIC to find optimal bandwidth in each direction. 
For comparison, we also fit model (\ref{eqn:bd-alt}) using classical backfitting in \citet{fan2005nonparametric}, smooth backfitting \citep{mammen2022backfitting,wsbackfit} with \emph{wsbackfit} package in \texttt{R}, and penalized splines approach with \emph{mgcv} package in \texttt{R} \citep{wood2015package,team2013r}.

Figure \ref{fig:bos-pr-500} shows the estimated additive functions along with the partial residuals. The simplified smooth backfitting algorithm estimates the additive functions $\widehat{m}_j$ as a sum of two functions i.e. $\widehat{m}_j^*$ and $\widehat{g}_j$, where $\widehat{m}_j^*$ is the purely nonparametric part and  $\widehat{g}_j$ is the parametric part corresponding to the eigenvectors of eigenvalue 1 of the smoother $\bm{H}_{1,j}^*$. For comparison,  Figure \ref{fig:bos-pr-200_cbmg} includes the fits from gam() function in \emph{mgcv} package in \texttt{R},  from sback() function in \emph{wsbackfit} package in \texttt{R},  and from the classical backfitting of \citet{fan2005nonparametric}. Figure \ref{fig:bos-pr-200_cbmg} shows that the fits from the these methods are very similar. From both Figures \ref{fig:bos-pr-500} and \ref{fig:bos-pr-200_cbmg} we find that the  additive components for all the variables except $RM$ exhibit the following parametric forms:
\begin{align}
  m_i(X_i) &= a_i + b_i X_i \qquad \text{     for } i=2,3,4.
  \label{eqn:bos-sempar}
\end{align}
 This confirms with the observations of \citet{opsomer1998fully} and \citet{fan2005nonparametric}.

We use the proposed GLR and LF statistics to test whether the semiparametric null model (\ref{eqn:bos-sempar}) holds against the additive alternative model (\ref{eqn:bd-alt}). For the LF test statistic, we consider the family of  LINEX loss functions (\ref{eqn:linex-def}) with parameters $(s=\{0,0.2,0.5,1\}, t=1)$. For comparison, we include the results for the  GLR test  in \citet{fan2005nonparametric}, which we refer as GLR(FJ). Further, we also include the results from the backfit test (SB) defined in \citet{mammen2022backfitting}. For convenience, the test statistic for SB method is computed as
\begin{align*}
 S_n &=    \sum_{j=2}^4 \sum_{i=1}^n \{ \widehat{m}_j(X_{(i)j})-\widehat{g}_j(X_{(i)j})\}^2 \widehat{f}_j(X_{(i)j}) (X_{(i)j}-X_{(i-1)j}), ~ \text{  with  }  ~ X_{(0)j}:=X_{(1)j},
\end{align*}
where $X_{(i)j}$ is the $i$th order statistic of $X_j$, $\widehat{g}_j(\cdot)$ is the corresponding parametric part. The null distributions of the test statistics $\lambda_n$,  $q_n$, $F_{\lambda}$, $F_{q}$, $\lambda_n(FJ)$, and $S_n$ are necessary to compute their $p$-values. Therefore, we use the conditional bootstrap method described in Supplementary Material to obtain the null distributions of the test statistics.  The optimal bandwidths  $\bm{h}_{\text{opt}}=(0.40,0.20,0.59,0.39)^T$ are computed using the procedure described in Section \ref{sec:simul}. 
\begin{figure}[h]
  \centering
  \makebox{\includegraphics[scale=0.3]{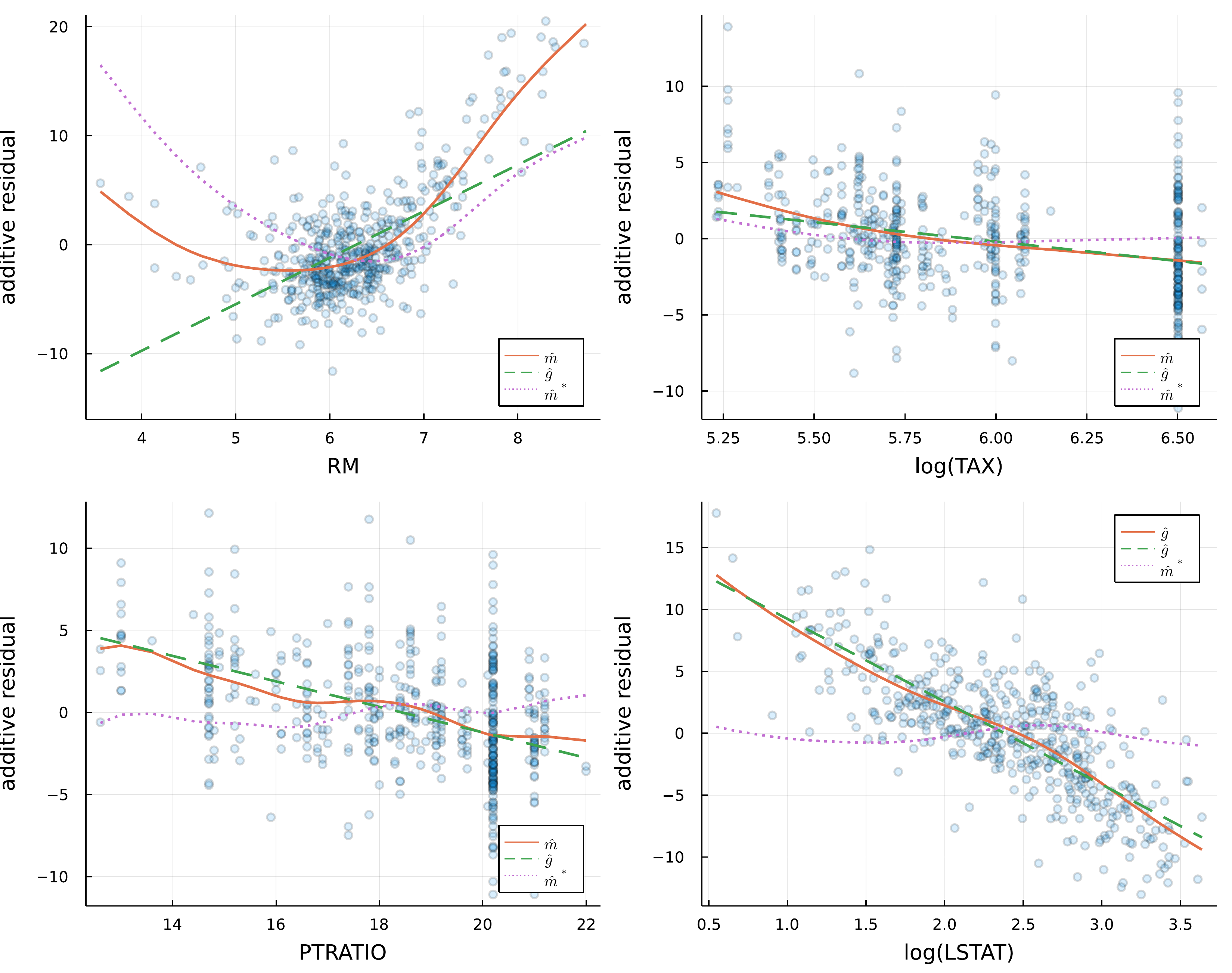}}
  \caption{Partial residual plots along with fitted regression curves for the Boston housing dataset.  The solid lines represent the estimated additive functions $\widehat{m}_j=\widehat{m}_j^*+\widehat{g}_j$, $j=1,2,3,4$; dotted lines indicate the purely nonparametric functions $\widehat{m}_j^*$, dashed lines represent the parametric part $\widehat{g}_j$, for $j=1,2,3,4$.}
  \label{fig:bos-pr-500}
\end{figure}
%
\begin{figure}[h]
  \centering
  \makebox{\includegraphics[scale=0.3]{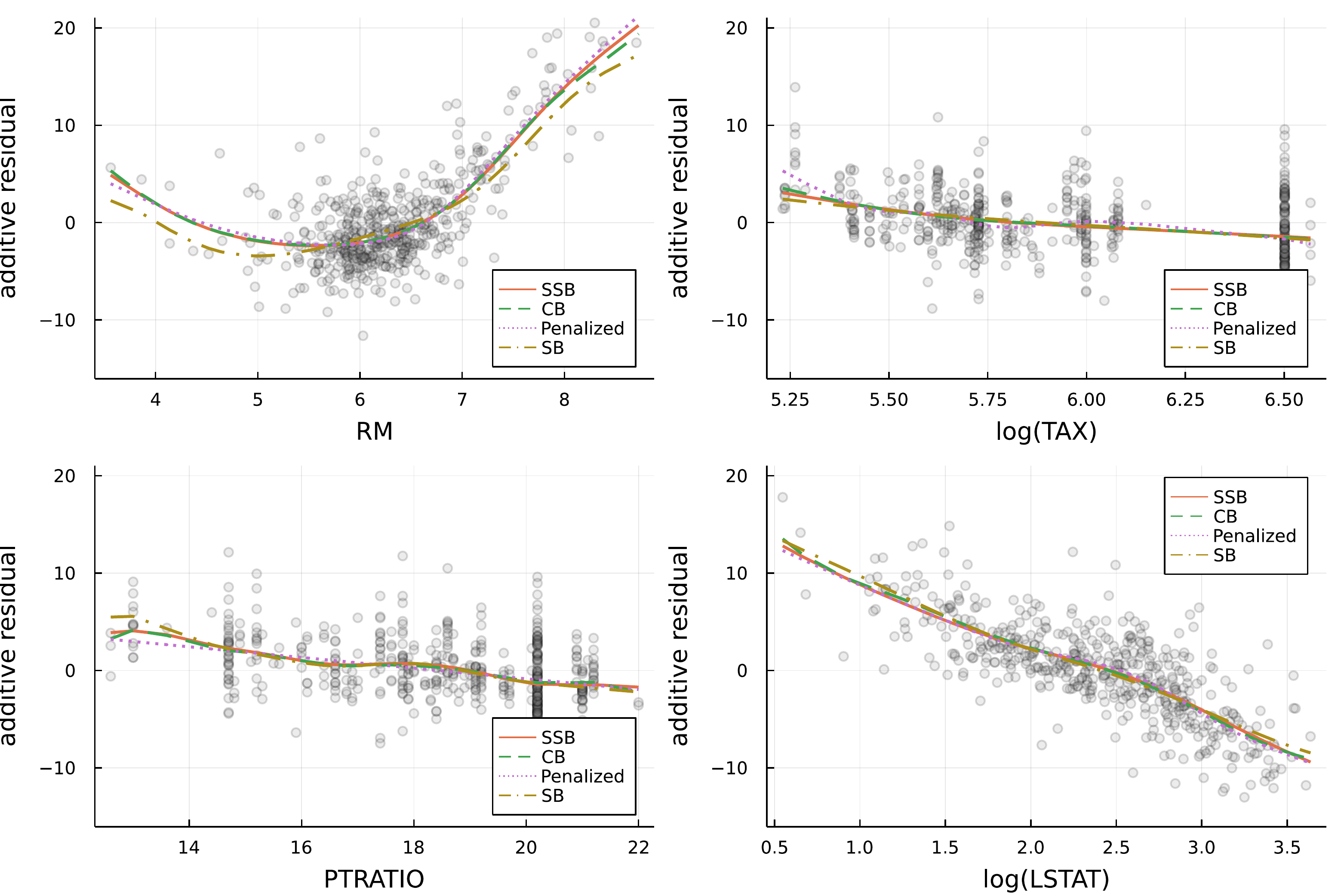}}
  \caption{Comparison of fits for the Boston housing dataset.  SSB: (solid)   $\widehat{m}_j$, $j=1,2,3,4$; Penalized: (dotted) fits from gam() function in \emph{mgcv} package in \texttt{R}; CB: (dashed) fits from the classical backfitting of \citet{fan2005nonparametric}; SB: (dash dot)  fits from  sback() function in \emph{wsbackfit} package in \texttt{R}.}
  \label{fig:bos-pr-200_cbmg}
\end{figure}

Table \ref{tab:bos-200} provides 
the $p$-values for statistics $\lambda_n$,  $q_n$, $F_{\lambda}$, $F_{q}$, $\lambda_n(FJ)$, and $S_n$ with the following five different bandwidths $(\frac{1}{2}\bm{h}_{\text{opt}}, \frac{2}{3}\bm{h}_{\text{opt}}, \bm{h}_{\text{opt}}, \frac{3}{2}\bm{h}_{\text{opt}}, 2\bm{h}_{\text{opt}})^T$ and using $1000$ bootstrap replications to compute null distributions. These results indicate that  the semiparametric model (\ref{eqn:bos-sempar}) is appropriate for this dataset within the additive models. For smaller bandwidths (undersmoothing) there is some evidence to reject the null hypothesis which is not surprising. For larger bandwidths, the estimated additive functions look more like parametric models and therefore the evidence is in favor of the null hypothesis. For the optimal bandwidths considered for estimation, the proposed GLR  and LF tests conclude that semiparametric additive model is appropriate at $0.01$ and $0.1$ significance levels, respectively. This result also validates our finding that the LF test is asymptotically more powerful than the GLR test. The $p$-values of the statistic $\lambda_n(FJ)$ are the smallest among all. We note  that the optimal bandwidths are computed using simplified smooth backfitting and the same are used for GLR(FJ) as well.

\begin{table}
  \caption{\label{tab:bos-200} P-values of statistics $\lambda_n$,  $q_n$, $F_{\lambda}$, $F_{q}$, $\lambda_n(FJ)$, and $S_n$ test statistics for a random sample of 200 observations from the Boston housing data. The LF test statistic uses the family of LINEX loss functions (\ref{eqn:linex-def}). We consider 1000 bootstrap replications to compute the null distributions of respective statistics.}\par
  \centering

  \begin{tabular}{|lccccccccc|}
      \hline
   Bandwidth &  $\lambda_n$ & $q_n(0,1)$ & $q_n(0.2,1)$ & $q_n(0.5,1)$ & $q_n(1,1)$ & $F_{\lambda}$ & $F_q$ & $\lambda_n(FJ)$ & $S_n$ \\
   \hline
   $\frac{1}{2}\bm{h}_{\text{opt}}$ & $0.001$ & $0.007$ & $0.01$ & $0.022$ & $0.024$ & $0.101$ & $0.125$ & $0.0$ & $0.044$ \\
   \hline
   $\frac{2}{3}\bm{h}_{\text{opt}}$ &  $0.005$ & $0.02$ & $0.032$ & $0.05$ & $0.047$ & $0.183$ & $0.232$ & $0.0$ & $0.237$ \\
   \hline
    $\bm{h}_{\text{opt}}$ &  $0.023$ & $0.09$ & $0.092$ & $0.104$ & $0.111$ & $0.218$ & $0.228$ &$0.01$ & $0.574$ \\
   \hline
    $\frac{3}{2}\bm{h}_{\text{opt}}$ & $0.08$ & $0.134$ & $0.15$ & $0.171$ & $0.205$ & $0.194$ & $0.155$ & $0.055$ & $0.418$ \\
   \hline
    $2\bm{h}_{\text{opt}}$ & $0.124$ & $0.081$ & $0.086$ & $0.102$ & $0.129$ & $0.15$ & $0.091$ &$0.079$ & $0.191$ \\                                       
      \hline
    \end{tabular}
  \end{table}

To sum up, the results in this section indicate that both GLR and LF tests are very useful in practical applications. While their performances  are sensitive to the choice of bandwidth, it is not straightforward to find optimal bandwidths for these statistics. Additionally, the finite sample performance of LF test is mildly sensitive to the choice of the loss function. Therefore, in practice, it is advisable to use both frameworks for a given hypothesis testing problem. This helps minimizing errors associated with hypothesis testing.    

\section{ Summary and Conclusions}

In this study, we develop a hypothesis testing framework for additive models using GLR and LF tests where simplified smooth backfitting is used for model estimation.  While the properties of GLR test are available in the literature for additive models estimated via classical backfitting \citep{opsomer2000asymptotic}, it is not the case for additive models estimated with simplified smooth backfitting \citep{huang2018}. Similarly, the results for the LF test are not available for additive models. We fill this void by proposing inference methods using GLR and LF tests when a model uses simplified smooth backfitting for estimation.  Under some regularity conditions, we show both the test statistics achieve Wilks phenomenon and have optimal power properties.  Furthermore, LF test is asymptotically more powerful than GLR test. This result is a new addition to the existing literature. The numerical performance of test statistics is also very similar across different bandwidths, and robust to different error distributions to some extent.

One possible direction for future research is to propose similar testing frameworks for generalized additive models. The LF test is asymptotically more powerful than the GLR test in linear additive models. It will be interesting to see whether the same result holds in generalized additive models.


\par
\begin{acknowledgements}
The author would like to thank the associate editor and two anonymous referees for their constructive feedback, which resulted in major changes to the paper's presentation. The  author would  also like to  thank Prof. Li-Shan Huang for offering a postdoctoral position and for sharing her research on simplified smooth backfitting, which provided the required framework for this article. 
The author gratefully acknowledges the support from grants 107-2811-M-007-014, 105-2118-M-007-006-MY2, and  107-2811-M-007-1047 by the Ministry Of Science and Technology (MOST) in Taiwan (R.O.C).
\end{acknowledgements}

\bibliographystyle{agsm}
\bibliography{SBL}

\newpage 

\begin{center}
\Huge{Supplementary Material}
\end{center}

\section*{Introduction}
This article develops a hypothesis testing framework for additive models. For a random sample $
\{Y_i,X_{i1},\ldots,X_{id} \}_{i=1}^n$, we consider
\begin{align}
  Y_i &= \alpha_0+ \sum_{j=1}^d m_j(X_{ij})+ \epsilon_i, \quad  i=1,\ldots,n,
  \label{eqn:int-model-suppl}
  \end{align}	
  where $\{\epsilon_i, i=1, \dots, n\}$ is a sequence of i.i.d. random variables with mean zero and
  finite variance $\sigma^2$. Each additive component function $m_j(\cdot)$, $j=1,\ldots,d$, is assumed to be an unknown smooth function and identifiable subject to the constraint, $E[m_j(\cdot)]=0$.  For simplicity in presentation,  the following hypothesis testing problem is considered 
  \begin{align}
    H_0: m_d(x_d)=0 \quad \text{vs.} \qquad   H_1: m_d(x_d) \ne 0, 
    \label{eqn:h1-suppl}
  \end{align}
  which tests whether the $d$th covariate is significant or not. 
  
  For readability, we repeat some notations and definitions that are provided in the main document. Let $\bm{m}_j=(m_j(X_{1j}),\ldots, m_j(X_{nj}))^T$ and $\bm{x}_j=(X_{1j},\ldots,X_{nj})^T$ for $j=1,\ldots,d$. Let  $\mathbb{X}_j=[\bm{1} ~ \bm{x}_j ~\cdots ~ \bm{x}_j^{p_j}]$ for $j=1,\ldots,d$, and $\mathbb{X}=[\bm{1} ~ \bm{x}_1 ~ \cdots ~\bm{x}_d ~ \cdots ~ \bm{x}_1^{p_1}~ \cdots ~ \bm{x}_d^{p_d}]$, where $\bm{1}$ is the vector of ones. Let $\mathbb{X}^{[-0]}=[\bm{x}_1 ~ \cdots ~ \bm{x}_d ~ \cdots ~ \bm{x}_1^{p_1}~ \cdots ~ \bm{x}_d^{p_d}]$ which is same as $\mathbb{X}$ without the column of ones. For any matrix $\bm{A}$, define $\bm{A}^{\perp}=\bm{I}-\bm{A}$ and $\bm{P}_{\bm{A}}=\bm{A}(\bm{A}^T\bm{A})^{-1}\bm{A}^T$.
  
  The following definitions are needed for the theoretical results.  Let $\mathcal{M}_1(\bm{H}_{p_j,j}^*)$ be a space spanned by the
  eigenvectors of $\bm{H}_{p_j,j}^*$ with eigenvalue 1. It includes polynomials of $\bm{x}_j$ up to $p_j$th order because $\bm{H}_{p_j,j}^* \bm{x}_j^k=\bm{x}_j^k$, $k=0,1\ldots,p_j$, and $j=1,\ldots,d$.  Suppose $\bm{G}$ is an orthogonal projection onto the space $\mathcal{M}_1(\bm{H}_{p_1,1}^*)+ \cdots + \mathcal{M}_1(\bm{H}_{p_d,d}^*)$ and $\bm{G}_j$ is an orthogonal projection onto the space $\mathcal{M}_1(\bm{H}_{p_j,j}^*)$, $j=1,\ldots,d$. Then,
  \begin{align}
    \bm{G} = \bm{P}_{\mathbb{X}} = \bm{P}_{\bm{1}} + \bm{P}_{\bm{P}_{\bm{1}}^{\perp} \mathbb{X}^{[-0]}}, \qquad \bm{G}_j = \bm{P}_{\mathbb{X}_j}, \label{eqn:back-gmat-def}
  \end{align}
  where $\bm{P}_{\mathbb{X}}=\mathbb{X}(\mathbb{X}^T\mathbb{X})^{-1}\mathbb{X}$ and $\bm{P}_{\bm{1}}$, $\bm{P}_{\mathbb{X}_j}$ and $\bm{P}_{\bm{P}_{\bm{1}}^{\perp} \mathbb{X}^{[-0]}}$ are defined similarly.  Let $\bm{G}_{[-d]}=\bm{P}_{\mathbb{X}^{[-d]}}$ where $\mathbb{X}^{[-d]}=[\bm{1} ~ \bm{x}_1 ~ \cdots ~ \bm{x}_{d-1} ~ \cdots ~ \bm{x}_1^{p_1} ~ \cdots ~ \bm{x}_{d-1}^{p_{d-1}}]$ and $\bm{x}_j^k=(X_{1j}^k, \ldots, X_{nj}^k)^T$ for  $k=0,1,\ldots,p_j$, as in (\ref{eqn:back-gmat-def}). 
  
  Define
  \begin{align}
  \bm{C} &= \bm{P}_{\bm{G}_{[-d]}^{\perp}\mathbb{X}_d^{[-0]}}+\bm{G}^{\perp}\bm{H}_{p_d,d}^*+\bm{H}_{p_d,d}^*\bm{G}^{\perp}-\bm{G}^{\perp} \bm{H}_{p_d,d}^* \bm{H}_{p_d,d}^* \bm{G}^{\perp} + O(n^{-1}h_d^{-1}\bm{I}+ n^{-1}\bm{J}),\label{eqn:cmat}\\
  \bm{D} &= \bm{G}^{\perp}- \sum_{j=1}^d \bigg\{ \bm{H}_{p_j,j}^*\bm{G}^{\perp} + O(n^{-1}h_j^{-1}\bm{I}+n^{-1}\bm{J})\bigg\}, \label{eqn:dmat}\\
  \bm{E} &= \bm{P}_{\bm{G}_{[-d]}^{\perp}\mathbb{X}_d^{[-0]}} + \bm{H}_{p_d,d}^* \bm{G}^{\perp}+O(n^{-1}h_d^{-1}\bm{I}+n^{-1}\bm{J}),\label{eqn:emat}
  \end{align}
  where $\bm{P}_{\bm{G}_{[-d]}^{\perp}\mathbb{X}_d^{[-0]}}=\bm{G}_{[-d]}^{\perp}\mathbb{X}_d^{[-0]} \left(\mathbb{X}_d^{{[-0]}^{T}} \bm{G}_{[-d]}^{\perp} \mathbb{X}_d^{[-0]} \right)^{-1}\mathbb{X}_d^{{[-0]}^{T}} \bm{G}_{[-d]}^{\perp}$, and  $\mathbb{X}_d^{[-0]}=[\bm{x}_d  \cdots  \bm{x}_d^{p_d}]$, $\bm{J}$ is the matrix of ones, and $\bm{I}$ is an identity matrix of size $n$.

  \subsection{Generalized Likelihood Ratio test:} The GLR test statistic is defined as
  \begin{align}
  \lambda_n(H_0) = [\ell(H_1)-\ell(H_0)] \approxeq \frac{n}{2} \log\frac{RSS_0}{RSS_1} \approx \frac{n}{2} \frac{RSS_0-RSS_1}{RSS_1},
  \label{eqn:glrt}
  \end{align}
  where $RSS_0$ and $RSS_1$ are the residual sum of squares under the null and alternative, respectively, and  reject the null hypothesis when $\lambda_n(H_0)$ is large. Analogously, the following F-type of test \citep{huang2010analysis} is also developed
  \begin{align}
    F_{\lambda} & = \frac{\bm{y}^T \bm{C} \bm{y}}{\bm{y}^T \bm{D} \bm{y}} \frac{\text{tr}(\bm{D})}{\text{tr}(\bm{C})}, \label{eqn:f-lambda}
  \end{align}
  where $\text{tr}(\cdot)$ denotes the trace.

  \subsection{Loss Function test:}
The LF test statistic is defined as
\begin{align}
q_n(H_0) &= \frac{Q_n}{n^{-1}RSS_1}\approx \frac{d''(0)/2 \sum_{i=1}^n(\widehat{m}_+(X_{i1},\ldots,X_{id})- \widetilde{m}_{+}^{(-d)}(X_{i1},\ldots,X_{i(d-1)}))^2+R}{n^{-1}RSS_1},
\label{eqn:lft}
\end{align}
where $\widehat{m}_{+}$ and $\widetilde{m}_{+}^{(-d)}$ are the fitted values of the models under null and alternative, respectively, $RSS_1$ is the residual sum of squares under alternative and $R$ is the remainder term in the Taylor expansion of $d(\cdot)$. We reject the null hypothesis  when $q_n(H_0)$ is large.
The corresponding F-type of test statistic is define as 
\begin{align}
  F_q &= \frac{\bm{y}^T\bm{E}^T\bm{E}\bm{y}}{\bm{y}^T\bm{D}\bm{y}} \frac{\text{tr}(\bm{D})}{\text{tr}(\bm{E}^T\bm{E})}, \label{eqn:f-q}
\end{align}
for $\bm{E}$ defined in (\ref{eqn:emat}).

\section{Assumptions} We repeat the assumptions that are outlined in the main document.
\begin{assumption}\label{as:1-x}  The 
	densities $f_j(\cdot)$ of $X_j$  are Lipschitz-continuous and bounded away
	from 0 and have bounded support $\Omega_j$ for $j=1,\ldots,d$. The joint density of $X_j$ and $X_{j'}$, $f_{j,j'}(\cdot,\cdot)$, for $1 \le j \neq j' \le d$, is also Lipschitz continuous and have bounded support.
\end{assumption}
\begin{assumption} \label{as:2-ker-h}  The kernel $K(\cdot)$ is a bounded
	symmetric density function with bounded support and satisfies Lipschitz
	condition. The bandwidth $h_j \rightarrow 0$   and $n h_j^2 / (\ln n)^2 \rightarrow \infty$, $j=1,\ldots,d$, as $n \rightarrow \infty$.
\end{assumption}
\begin{assumption} \label{as:3-m}
	The $(2p_j+2)-$th derivative of $m_j(\cdot)$, $j=1,\ldots,d$, exists. 
\end{assumption}

\begin{assumption} \label{as:4-e}
	The error $\epsilon$ has mean 0, variance
	$\sigma^2$, and  finite fourth moment.
\end{assumption}

\begin{assumption}\label{as:5-d}
	The loss function $d: \mathbb{R} \rightarrow \mathbb{R}^+$ has a unique minimum at $0$, and
	$d(z)$ is monotonically nondecreasing as $|z| \rightarrow \infty$.
	Furthermore, $d(z)$ is twice continuously differentiable at $0$ with $d(0)=0$,
	$d'(0)=0$, $M=\frac{1}{2}d''(0) \in (0,\infty)$, and $|d''(z)-d''(0)| \le
	C|z|$ for any $z$ near $0$.
\end{assumption}

 \section{Required Lemmas and Proofs}

 The explicit expressions for the estimators $\widehat{\bm{m}}_j^*$, $j=1,\ldots,d$ are provided as follows. Let
 \begin{align}
  \bm{A}_j=(\bm{I}-(\bm{H}_{p_j,j}^*- \bm{G}_j))^{-1}(\bm{H}_{p_j,j}^*- \bm{G}_j)~ \text{ and }  \bm{A}=\sum_{j=1}^d\bm{A}_j. \label{eqn:def_aj}
 \end{align}
  By Proposition 3 in \citet{buja1989linear}, we obtain $\widehat{\bm{m}}_j^* = \bm{A}_j \left(\bm{I}+\bm{A}\right)^{-1}\bm{G}^{\perp}\bm{y}$. Therefore, the fitted response under the alternative can be written as 
  \begin{align}
    \widehat{\bm{y}}(H_1) &= \sum_{j=1}^d \widehat{\bm{m}}_j^* + \bm{G}\bm{y}=  \left(\sum_{j=1}^d \bm{A}_j \left(\bm{I}+\bm{A}\right)^{-1}\bm{G}^{\perp} + \bm{G} \right) \bm{y} :=  \bm{W} \bm{y}. \label{eqn:w}
  \end{align}
  Analogously, the fitted response under the null can be written as 
\begin{align}
  \widehat{\bm{y}} (H_0) &= \left(\sum_{j=1}^{d-1} \bm{A}_j \left(\bm{I}+\bm{A}\right)^{-1}\bm{G}_{[-d]}^{\perp} + \bm{G}_{[-d]} \right) \bm{y} :=  \bm{W}^{[-d]} \bm{y}. \label{eqn:w-d}
\end{align}

The following lemma simplifies the expressions for the $RSS_0-RSS_1$.
\begin{lemma} \label{lem:an12}
	Denote $A_{n1}=(\bm{I}-\bm{W}^{[-d]})^T(\bm{I}-\bm{W}^{[-d]})$ and
	$A_{n2}=(\bm{I}-\bm{W})^T(\bm{I}-\bm{W})$ where $\bm{W}$  and $\bm{W}^{[-d]}$ are defined in (\ref{eqn:w}) and (\ref{eqn:w-d}), respectively. If assumptions
	\ref{as:1-x}--\ref{as:3-m} hold, then
	\begin{align}
	RSS_0-RSS_1 &= \bm{y}^T(A_{n1}-A_{n2})\bm{y}
	\end{align}
	and
	\begin{align}
	A_{n1}-A_{n2} &= \bm{P}_{\bm{G}_{[-d]}^{\perp}\mathbb{X}_d^{[-0]}}+\bm{G}^{\perp}\bm{H}_{p_d,d}^*+\bm{H}_{p_d,d}^*\bm{G}^{\perp}-\bm{G}^{\perp} \bm{H}_{p_d,d}^* \bm{H}_{p_d,d}^* \bm{G}^{\perp} + O(n^{-1}h_d^{-1}\bm{I}+ n^{-1}\bm{J}),
  \end{align}
  where $\bm{P}_{\bm{G}_{[-d]}^{\perp}\mathbb{X}_d^{[-0]}}=\bm{G}_{[-d]}^{\perp}\mathbb{X}_d^{[-0]} \left(\mathbb{X}_d^{{[-0]}^{T}} \bm{G}_{[-d]}^{\perp} \mathbb{X}_d^{[-0]} \right)^{-1}\mathbb{X}_d^{{[-0]}^{T}} \bm{G}_{[-d]}^{\perp}$, and  $\mathbb{X}_d^{[-0]}=[\bm{x}_d  \cdots  \bm{x}_d^{p_d}]$, $\bm{J}$ is the matrix of ones, and $\bm{I}$ is an identity matrix of size $n$.
\end{lemma}
\begin{proof}
  As shown in \citet{huang2018}, the diagonal elements of $\bm{H}_{p_j,j}^*$ and $\bm{H}_{p_j,j}^*\bm{H}_{p_j,j}^*$, $j=1,\ldots,d$, are of order $O(n^{-1}h_j^{-1})$ and the off-diagonal elements are of order $O(n^{-1})$. Similarly, the elements of $\bm{H}_{p_j,j}^*\bm{H}_{p_l,l}^*$ are of order $O(n^{-1})$ for $j \neq l$. Since the elements of $\bm{H}_{p_j,j}^*$ are of smaller order, we can write
  \begin{align}
    \begin{split}
    \bm{A}_j &= \left[\bm{I}-(\bm{H}_{p_j,j}^*-\bm{G}_j)\right]^{-1}(\bm{H}_{p_j,j}^*-\bm{G}_j) =\bm{H}_{p_j,j}^*-\bm{G}_j + O(n^{-1}h_j^{-1}\bm{I}+n^{-1}\bm{J}),\\
    \bm{A}(\bm{I}+\bm{A})^{-1} &= \left[\sum\bm{A}_j \right] \left[\bm{I}+\sum\bm{A}_j\right]^{-1}= 
    \sum_{j=1}^d \left\{ \bm{H}_{p_j,j}^*-\bm{G}_j + O(n^{-1}h_j^{-1}\bm{I}+n^{-1}\bm{J}) \right\},  
    \end{split}
    \label{eqn:lem1-a}
  \end{align}
  where $\bm{I}$ is the identity matrix and $\bm{J}$ is the matrix of 1's of size $n$. Since $\bm{G}_j \in \mathcal{M}_1(\bm{H}_{p_j,j}^*)$, it follows that $\bm{H}_{p_j,j}^*\bm{G}_j=\bm{G}_j$. Therefore,
  \begin{align*}
    \bm{H}_{p_j,j}^*-\bm{G}_j = \bm{H}_{p_j,j}^*(\bm{I}-\bm{G}_j)= \bm{H}_{p_j,j}^*\bm{G}_j^{\perp}.
  \end{align*}   
  Consequently, we write (\ref{eqn:w}) as 
  \begin{align}
    \bm{W} &= \bm{A}(\bm{I}+\bm{A})^{-1} \bm{G}^{\perp} + \bm{G} \nonumber \\
    &= \sum_{j=1}^d \left\{ \bm{H}_{p_j,j}^*\bm{G}_j^{\perp} + O(n^{-1}h_j^{-1}\bm{I}+n^{-1}\bm{J}) \right\} \bm{G}^{\perp} + \bm{G} \nonumber \\
    &= \sum_{j=1}^d \left\{ \bm{H}_{p_j,j}^*\bm{G}^{\perp} + O(n^{-1}h_j^{-1}\bm{I}+n^{-1}\bm{J}) \right\}  + \bm{G},
    \label{lem1:w-s1}
  \end{align}
  where the last step uses the fact that $\bm{G}_j^{\perp}\bm{G}^{\perp}=\bm{G}^{\perp}$. Let $\bm{G}_{[-d]}$ be the parametric projection matrix of the first $d-1$ components defined similar to (\ref{eqn:back-gmat-def}). Based on the properties of the projection matrices, we obtain
  \begin{align}
    \bm{G} &= \bm{G}_{[-d]} + \bm{P}_{\bm{G}_{[-d]}^{\perp}\mathbb{X}_d^{[-0]}}, \nonumber \\
    \bm{G}^{\perp} &= \bm{G}_{[-d]}^{\perp} -\bm{P}_{\bm{G}_{[-d]}^{\perp}\mathbb{X}_d^{[-0]}},
    \label{lem1:g-ortho}
  \end{align}
  where $\bm{P}_{\bm{G}_{[-d]}^{\perp}\mathbb{X}_d^{[-0]}}$ defined in (\ref{eqn:emat}).  
    Combination of  (\ref{lem1:w-s1}) and (\ref{lem1:g-ortho}) and some rearrangement of terms yields
  \begin{align*}
    \bm{I}-\bm{W} &= \bm{G}^{\perp}- \left[ \sum_{j=1}^{d-1} \left\{ \bm{H}_{p_j,j}^* \bm{G}_{[-d]}^{\perp} + O(n^{-1}h_j^{-1}\bm{I}+n^{-1}\bm{J}) \right\} + \bm{H}_{p_d,d}^*\bm{G}^{\perp} + O(n^{-1}h_d^{-1}\bm{I}+ n^{-1}\bm{J}) \right]. 
  \end{align*}
  Observe that $\bm{H}_{p_j,j}^* \bm{G}_{[-d]}^{\perp}=\bm{H}_{p_j,j}^* \bm{G}_{j}^{\perp} \bm{G}_{[-d]}^{\perp}$, for $j=1,\ldots, d-1$.  The elements of $\bm{G}_{j}^{\perp} \bm{H}_{p_j,j}^* \bm{H}_{p_j,j}^* \bm{G}_{j}^{\perp}$ are of smaller order than the elements of  $\bm{H}_{p_j,j}^* \bm{G}_{j}^{\perp}$ since the latter has eigenvalues in $[0,1)$. Therefore,
  \begin{align*}
    (\bm{I}&-\bm{W})^T(\bm{I}-\bm{W})\\ &= \bm{G}^{\perp} -\bigg[ \sum_{j=1}^{d-1} \left\{  \bm{G}_{[-d]}^{\perp}\bm{H}_{p_j,j}^* +\bm{H}_{p_j,j}^*\bm{G}_{[-d]}^{\perp} -\bm{G}_{[-d]}^{\perp} \bm{H}_{p_j,j}^* \bm{H}_{p_j,j}^* \bm{G}_{[-d]}^{\perp}  + O(n^{-1}h_j^{-1}\bm{I}+n^{-1}\bm{J}) \right\} \\ & \qquad + \bm{G}^{\perp}\bm{H}_{p_d,d}^*+\bm{H}_{p_d,d}^*\bm{G}^{\perp} -\bm{G}^{\perp} \bm{H}_{p_d,d}^* \bm{H}_{p_d,d}^* \bm{G}^{\perp} + O(n^{-1}h_d^{-1}\bm{I}+ n^{-1}\bm{J}) \bigg]. 
  \end{align*}
  Similar computations yield
  \begin{align*}
    (\bm{I}-&\bm{W}^{[-d]})^T(\bm{I}  -\bm{W}^{[-d]}) \\ &= \bm{G}_{[-d]}^{\perp}- \bigg[ \sum_{j=1}^{d-1} \left\{\bm{G}_{[-d]}^{\perp}\bm{H}_{p_j,j}^*+ \bm{H}_{p_j,j}^* \bm{G}_{[-d]}^{\perp}-\bm{G}_{[-d]}^{\perp} \bm{H}_{p_j,j}^* \bm{H}_{p_j,j}^* \bm{G}_{[-d]}^{\perp}  + O(n^{-1}h_j^{-1}\bm{I}+n^{-1}\bm{J}) \right\} \bigg]. 
  \end{align*}
Hence,
\begin{align*}
	A_{n1}-A_{n2} &=\bm{P}_{\bm{G}_{[-d]}^{\perp}\mathbb{X}_d^{[-0]}}+\bm{G}^{\perp}\bm{H}_{p_d,d}^*+\bm{H}_{p_d,d}^*\bm{G}^{\perp}-\bm{G}^{\perp} \bm{H}_{p_d,d}^* \bm{H}_{p_d,d}^* \bm{G}^{\perp} + O(n^{-1}h_d^{-1}\bm{I}+ n^{-1}\bm{J}).
  \end{align*}

\end{proof}

\begin{lemma}\label{lem:d1n}
 If assumptions \ref{as:1-x}--\ref{as:4-e} hold, then under $H_0:\bm{m}_d=0$
          \begin{align}
            d_{1n} &\equiv \bm{m}_+^T(A_{n1}-A_{n2})\bm{m}_+ + 2\epsilon^T(A_{n1}-A_{n2})\bm{m}_+\nonumber\\&=O_p\left(1+\sum_{j=1}^d nh_j^{4(p_j+1)}+\sum_{j=1}^d\sqrt{n}h_j^{2(p_j+1)}\right),
            \label{eqn:lem2-res}
          \end{align}
          where $A_{n1}$ and $A_{n2}$ are defined in Lemma \ref{lem:an12} and $\bm{m}_+=\bm{m}_1+\ldots+\bm{m}_d$. 
        \end{lemma}

\begin{proof}
          From \cite{hua14local}, we have $\bm{H}_{p_j,j}^* \bm{m}_j= \bm{m}_j + \bm{1} \cdot O_p(h_j^{2(p_j+1)})$ for $p_j =0, 1,2,3$,  and
          $\bm{G}_{j}\bm{m}_{j}= \bm{1} \cdot O_p(1/\sqrt{n})$ for $j=1,\ldots,d$, where $\bm{1}$ is the vector of ones. The calculations analogous to Lemma \ref{lem:an12} yield, under $H_0:\bm{m}_d=0$, that
          \begin{align}
            (\bm{I}-\bm{W})\bm{m}_+ &= \left(\bm{I}- \bm{G} -\sum_{j=1}^d\left\{\bm{H}_{p_j,j}^* \bm{G}^{\perp}+O(n^{-1}h_j^{-1}\bm{I}+n^{-1}\bm{J}) \right\} \right) \bm{m}_+ \nonumber \\
                                    &= \bm{m}_+-\bm{m}_+ + \bm{1} \cdot O_p\left(\sum_{j=1}^d h_j^{2(p_j+1)} \right)+ \bm{1} \cdot O_p \left( 1/\sqrt{n} \right)  \nonumber \\
            &= \bm{1} \cdot O_p\left(\sum_{j=1}^d h_j^{2(p_j+1)} \right)+ \bm{1} \cdot O_p \left( 1/\sqrt{n} \right).
            \label{eqn:lem2-mm}
          \end{align}
          Consequently,
          \begin{align}
            \bm{m}_+^T(\bm{I}-\bm{W})^T(\bm{I}-\bm{W})\bm{m}_+=O_p\left(1+\sum_{j=1}^d nh_j^{4(p_j+1)}\right), 
\\
            \bm{m}_+^T(\bm{I}-\bm{W}^{[-d]})^T(\bm{I}-\bm{W}^{[-d]})\bm{m}_+=O_p\left(1+ \sum_{j=1}^{d-1} nh_j^{4(p_j+1)}\right). \nonumber
                      \end{align}
          Moreover,
          \begin{align*}
            (\bm{I}-\bm{W})\bm{\epsilon} &= \bm{\epsilon} + \bm{1}. o_p\left(1\right)
          \end{align*}
          which implies that under assumption \ref{as:4-e}
          \begin{align}
            \bm{\epsilon}^T(\bm{I}-\bm{W})^T(\bm{I}-\bm{W})\bm{m}_+=O_p(1+\sum_{j=1}^d\sqrt{n}h_j^{2(p_j+1)}).
            \label{eqn:lem2-em}
          \end{align}
          Hence, the stated result (\ref{eqn:lem2-res}) follows from (\ref{eqn:lem2-mm}) and (\ref{eqn:lem2-em}).
        \end{proof}

        \begin{theorem} \label{thm:glr1} (GLR test)
          Suppose that conditions \ref{as:1-x}--\ref{as:4-e} hold and $0 \le p_j \le 3$, $j=1,\ldots,d$. Then, under
          $H_0$ for the testing problem (\ref{eqn:h1})
          \begin{align}
          P\left\{ \sigma_{n}^{-1}\left(  \lambda_n(H_0)-\mu_n-\frac{1}{2\sigma^2}d_{1n} \right) < t |\mathcal{X} \right\} \xrightarrow{d} \bm{\Phi}(t),
          \label{eqn:th1-glr-n}
          \end{align}
          where  $d_{1n}=O_p\left(1+\sum_{j=1}^dnh_j^{4(p_j+1)}+\sum_{j=1}^d\sqrt{n}h_j^{2(p_j+1)}\right)$ and $\bm{\Phi}(\cdot)$ is the standard normal distribution. Furthermore, if $nh_j^{4(p_j+1)}h_d \rightarrow 0$ for $j=1,\ldots,d$, conditional on the sample space $\mathcal{X}$, 
          $ r_k\lambda_n(H_0) \xrightarrow{} \chi^2_{r_k\mu_n}$ as $n \rightarrow \infty$.
          Similarly,
          \begin{align}
          F_{\lambda} &= \frac{2\lambda_n(H_0) \text{tr}(\bm{D})}{n \text{tr}(\bm{C})} \xrightarrow{} F_{\text{tr}(C),\text{tr}(D)},
          \label{eqn:th1-glr-f}
          \end{align}
          as $n \rightarrow \infty$, where $\text{tr}(C)$ and $\text{tr}(D)$ are the degrees of freedom.
        \end{theorem}

\begin{proof}: Recall that 
  \begin{align}
    \lambda_n(H_0) \approx \frac{n}{2} \frac{RSS_0-RSS_1}{RSS_1}.
    \end{align}
     
\underline{\textbf{Proof of \ref{eqn:th1-glr-n}}:}

\vspace{1em}
\textbf{(i) Asymptotic Expression for $RSS_0-RSS_1$:}

Using the notation from Lemma \ref{lem:an12}, we write 
            \begin{align}
              RSS_0-RSS_1 &= \bm{y}^T(A_{n1}-A_{n2})\bm{y} \nonumber \\
                          &= \bm{\epsilon}^T(A_{n1}-A_{n2})\bm{\epsilon} + \left[ \bm{m}_+^T(A_{n1}-A_{n2})\bm{m}_+ + 2\bm{\epsilon}^T(A_{n1}-A_{n2})\bm{m}_+ \right] \nonumber\\
                          &=\bm{\epsilon}^T\bm{C}\bm{\epsilon} + d_{1n},
                            \label{eqn:th1-rs01}
            \end{align}
            where 
            \begin{align} 
              \bm{C} &= A_{n1}-A_{n2} \nonumber \\ &= \bm{P}_{\bm{G}_{[-d]}^{\perp}\mathbb{X}_d^{[-0]}}+\bm{G}^{\perp}\bm{H}_{p_d,d}^*+\bm{H}_{p_d,d}^*\bm{G}^{\perp}-\bm{G}^{\perp} \bm{H}_{p_d,d}^* \bm{H}_{p_d,d}^* \bm{G}^{\perp} + O(n^{-1}h_d^{-1}\bm{I}+ n^{-1}\bm{J}) \nonumber\\
              &=(c_{ij})_{1\le i,j \le n}, \label{eq:glr-c} 
            \end{align}
            and $d_{1n}=\bm{m}_+^T(A_{n1}-A_{n2})\bm{m}_+ + 2\bm{\epsilon}^T(A_{n1}-A_{n2})\bm{m}_+$.
            With the help of Lemma \ref{lem:d1n}, we can bound the bias term $d_{1n}$ by
            $O_p\left(1+\sum_{k=1}^dnh_k^{4(p_k+1)}+\sum_{k=1}^d\sqrt{n}h_k^{2(p_k+1)}\right)$. We write
            \begin{align}
              \bm{\epsilon}^T\bm{C}\bm{\epsilon} &= \sum_{i=1}^n \epsilon_i^2 c_{ii} +\sum_{i \ne j}^n \epsilon_i\epsilon_j c_{ij} 
              =L_1+L_2. \label{eqn:th1-l1l2}
            \end{align}
            Since the leading terms of $\bm{C}$ in (\ref{eq:glr-c}) come from $\bm{H}_{p_d,d}^*$,  we obtain
            $c_{ii}=O(n^{-1}h_d^{-1}+n^{-1})$. Combination of  Assumption \ref{as:4-e} and Chebyshev inequality yields $L_1=\sigma^2
            E(\sum_{i=1}^nc_{ii})+ O_p(1/\sqrt{n}h_d)$. After some algebra, 
            \begin{align*}
E(\sum_{i=1}^nc_{ii}) &=\frac{2| \Omega_d |}{h_d} \bigg(\sum_{l=0}^{p_d}\sum_{m=0}^{p_d} v_{l+m} s^{(m+1),(l+1)} \\ & \qquad - \frac{1}{2} \int \bigg\{ \sum_{l=0}^{p_d}\sum_{m=0}^{p_d} (K_l*K_m)(u) (-1)^m s^{(m+1),(l+1)} \bigg\}^2 du  \bigg)+ o_p(h_d^{-1}), \end{align*}
where $| \Omega_d |$ is the length of the support of the density $f_d(x_d)$ of $X_d$. It remains to show that $L_2$
            converges to normal in distribution. Note that $E[L_2]=0$ and
            \begin{align*}
              \text{Var}(L_2 | \mathcal{X}) = \text{var}\left( \sum_{i \neq j}^n \epsilon_i\epsilon_j c_{ij} \right) = 4\sigma^4 \sigma_{n}^2,              \end{align*}
              where 
              \begin{align*}
              \sigma_{n}^2&= \sum_{i<j}c_{ij}^2= \frac{| \Omega_d |}{h_d} \int \bigg\{ \sum_{l=0}^{p_d}\sum_{m=0}^{p_d} (K_l*K_m)(u) (-1)^m s^{(m+1),(l+1)} \\&\qquad -\frac{1}{2} \int \bigg[ \sum_{l=0}^{p_d} \sum_{m=0}^{p_d} (K_l * K_m) (u+v) (-1)^m s^{(m+1),(l+1)} \bigg] \\ & \qquad \qquad \times \bigg[ \sum_{l=0}^{p_d} \sum_{m=0}^{p_d} (K_l * K_m) (v) (-1)^m s^{(m+1),(l+1)} \bigg] dv \bigg\}^2 du + o_p(h_d^{-1}).
              \end{align*}
            Application of Proposition 3.2 of \citet{de1987central} yields
            \begin{align}
              \frac{1}{2\sigma^2} \sigma_{n}^{-1} L_2 | \mathcal{X} \xrightarrow{d} N(0,1).
              \label{eqn:th1-clt}
            \end{align}

            \textbf{(ii) Asymptotic Expression for $RSS_1/n$:} By the definition of
            $RSS_!$,
            \begin{align*}
              RSS_1 &= \bm{\epsilon}^T A_{n2} \bm{\epsilon} + \bm{m}_+^T A_{n2}\bm{m}_+ + 2\bm{\epsilon}^TA_{n2}\bm{m}_+\\
              &=\bm{\epsilon}^T A_{n2} \bm{\epsilon} +d_{0n}.
            \end{align*}
            The arguments analogous to Lemma \ref{lem:d1n} yields
            \begin{align}
              d_{0n} &= \bm{m}_+^T A_{n2}\bm{m}_+ + 2\bm{\epsilon}^TA_{n2}\bm{m}_+ \nonumber \\
                     &= O_p\left(1+\sum_{k=1}^d nh_k^{4(p_k+1)}+\sum_{k=1}^d \sqrt{n}h_k^{2(p_k+1)}\right)\nonumber.
            \end{align}
            Note that, under the condition \ref{as:2-ker-h}, 
            $d_{0n}/n=o_p(1)$. Thus, it remains to show that
            $n^{-1}\bm{\epsilon}^TA_{n2}\bm{\epsilon}=\sigma^2 + o_p(1)$. From
            the proof of Lemma $\ref{lem:an12}$, we obtain
            \begin{align*}
              n^{-1}\bm{\epsilon}^TA_{n2}\bm{\epsilon} &=n^{-1}\bm{\epsilon}^T\left(\bm{I}- \bm{G}- \left(\sum_{j=1}^d \left\{ \bm{H}_{p_j,j}^*\bm{G}^{\perp} + O(n^{-1}h_j^{-1}\bm{I}+n^{-1}\bm{J}) \right\}  \right) \right)\bm{\epsilon}+o_p(1) \\
               &=n^{-1}\sum_{i=1}^n\epsilon_i^2  +o_p(1)\\
              &= \sigma^2 +o_p(1),
            \end{align*}
            which follows from the Chebyshev inequality and using the arguments
            analogous to the derivation of variance for (\ref{eqn:th1-l1l2}).

            \textbf{(iii) Conclusion :} By part(i), part(ii) and definition of
            $\lambda_n(H_0)$, we have
            \begin{align*}
              \lambda_n(H_0) &\approxeq  \frac{RSS_0-RSS_1}{2RSS_1/n} \\
                             &= \frac{d_{1n}+L_1+L_2}{2\sigma^2}\\
                             &= \frac{d_{1n}+ \sigma^2E(\sum_{i=1}^nc_{ii})+ L_2}{2\sigma^2}+o_p(h_d^{-1})\\
              &\approxeq \frac{d_{1n}}{2\sigma^2} + \mu_n + \frac{L_2}{2\sigma^2},
            \end{align*}
            where $\mu_n=E(\sum_{i=1}^nc_{ii})/2$.
            Therefore,  (\ref{eqn:th1-clt}) implies
            \begin{align*}
              P\left\{ \sigma_n^{-1}\left(  \lambda_n(H_0)-\mu_n-\frac{1}{2\sigma^2}d_{1n}\right) < t | \mathcal{X}\right\} \xrightarrow{d} \Phi(t).
            \end{align*}
           If $nh_k^{4(p_k+1)}h_d \rightarrow 0$ for $k=1,\ldots,d$, then
           $d_{1n}=o_p(h_d^{-1})$ which is dominated by $\mu_n=O(h_d^{-1})$. Then $r_k\lambda_n(H_0)|\mathcal{X} \xrightarrow{} \chi^2_{r_k\mu_n}$ as $n \rightarrow \infty$.
            
 \vspace{2em}

            \underline{\textbf{Proof of  (\ref{eqn:th1-glr-f}):}}

            By virtue of Lemma \ref{lem:an12}, the GLR test statistic is defined as
            \begin{align*}
              \lambda_n(H_0) &\approxeq \frac{n\bm{y}^T(A_{n1}-A_{n2})\bm{y}}{2 \bm{y}^TA_{n2}\bm{y}} \\
                             &=\frac{n\bm{y}^T\bm{C}\bm{y}}{2 \bm{y}^T \bm{D} \bm{y}},
            \end{align*}
            for $\bm{D}=\bm{G}^{\perp}- \left(\sum_{j=1}^d \left\{ \bm{H}_{p_j,j}^*\bm{G}^{\perp} + O(n^{-1}h_j^{-1}\bm{I}+n^{-1}\bm{J})\right\}\right)$ and $\bm{C}$ defined in (\ref{eq:glr-c}).
            As discussed in \citet{huang2010analysis}, for $F-$type statistics,
            \begin{align}
              F &=\frac{\bm{y}^T\bm{C}\bm{y}/\text{tr}(\bm{C})}{\bm{y}^T\bm{D}\bm{y}/\text{tr}(\bm{D})},
              \end{align}
            the $F-$ distribution is warranted if $\bm{C}$ and $\bm{D}$ are both
            projection matrices (symmetric and idempotent) and they are
            orthogonal to each other. Clearly, both $\bm{C}$ and $\bm{D}$ are not projection matrices and not orthogonal to each other. However, following \citet{hua08analysis}, we show these properties hold asymptotically.
            It is straightforward to show both  $\bm{C}$ and $\bm{D}$ are asymptotically
            idempotent. Now it remains to show that they are asymptotically orthogonal. 
            Observe
            \begin{align*}
              E\{\bm{C} \bm{D} \bm{y} | \mathcal{X}\} &= \bm{C} \left(\sum_{j=1}^d \left[O(h_j^{2(p_j+1)}) + O_p(1/\sqrt{nh_j}+1/\sqrt{n}) \right] \right) =o(1).
            \end{align*}
            Based on the definition of asymptotic orthogonality in \citet{hua08analysis}, we claim  $\bm{C}$
            and $\bm{D}$ are asymptotically orthogonal. Therefore,
            \begin{align*}
              F &=\frac{\bm{y}^T\bm{C}\bm{y}/\text{tr}(\bm{C})}{\bm{y}^T\bm{D}\bm{y}/\text{tr}(\bm{D})} 
                = \frac{2\lambda_n(H_0)}{n} \frac{\text{tr}(\bm{D})}{\text{tr}(\bm{C})}  = \frac{2\lambda_n(H_0)}{\text{tr}(\bm{C})}, 
            \end{align*}
            because $\text{tr}(\bm{D})/n=
            \frac{1}{n}\text{tr}\left(\bm{G}^{\perp}- \left(\sum_{j=1}^d \left\{ \bm{H}_{p_j,j}^*\bm{G}^{\perp} + O(n^{-1}h_j^{-1}\bm{I}+n^{-1}\bm{J})\right\}\right)\right) \xrightarrow{} 1 $ as $nh_d \rightarrow
            \infty$ and $n \rightarrow \infty$.
        \end{proof}


%
\begin{theorem}\label{thm:loss1} (LF test)
  Suppose that conditions \ref{as:1-x}--\ref{as:5-d} hold and $0 \le p_j \le 3$, $j=1,\ldots,d$. Then, under
  $H_0$ for the testing problem (\ref{eqn:h1})
  \begin{align}
  P\left\{ \delta_n^{-1} \left(  \frac{q_n(H_0)}{M}-\nu_n-\frac{1}{\sigma^2}b_{1n} \right) < t |\mathcal{X} \right\} \xrightarrow{d} \bm{\Phi}(t), \label{eqn:th2-n}
  \end{align}
  where  $b_{1n}=O_p\left(1+\sum_{j=1}^d nh_j^{4(p_j+1)}\right)$. Furthermore, if $nh_j^{4(p_j+1)}h_d \rightarrow 0$ for $j=1,\ldots,d$, conditional on $\mathcal{X}$, 
  $ s_k M^{-1}q_n(H_0) \xrightarrow{} \chi^2_{s_k\nu_n}$ as $n \rightarrow \infty$.
  Similarly,
  \begin{align}
  F_q &= \frac{q_n(H_0) \text{tr}(\bm{D})}{M n \text{tr}(\bm{E}^T\bm{E})} \xrightarrow{} F_{\text{tr}(\bm{E}^T\bm{E}),\text{tr}(\bm{D})}, \label{eqn:th2-f}
  \end{align}	
  as $n \rightarrow \infty$.
\end{theorem}

\begin{proof} \underline{\textbf{Proof of (\ref{eqn:th2-n}):}}

          Consider the LF test statistic in (\ref{eqn:lft})
         \begin{align*}
           q_n(H_0) &= \frac{Q_n}{n^{-1}SSR1}=\frac{\sum_{i=1}^n d\left\{\sum_{j=1}^n e_{ij}Y_j\right\}}{n^{-1} RSS_1},
         \end{align*}
         where  $d(\cdot)$ is the loss function defined in Assumption \ref{as:5-d} and $e_{ij}$ is the $(i, j)$th, $1 \le i,j \le n$, element of $ \bm{P}_{\bm{G}_{[-d]}^{\perp}\mathbb{X}_d^{[-0]}}+\bm{H}_{p_d,d}^* \bm{G}^{\perp}+O(n^{-1}h_d^{-1}\bm{I}+n^{-1}\bm{J})$. The arguments analogous to Lemma
         \ref{lem:an12} yield that , under $H_0$,
         \begin{align*}
           \sum_{j=1}^n e_{ij}m_{+j} &=  O\left(\sum_{k=1}^d h_k^{2(p_k+1)}\right) +  O_p\left(1/\sqrt{n}\right).
         \end{align*}
         Note that the dominant orders for the elements $e_{ij}$'s come from $\bm{H}_{p_d,d}^*$. Therefore, diagonal elements $e_{ii}$'s are of order $O(n^{-1}h_d^{-1}+n^{-1})$ and the off-diagonal elements $e_{ii'}$, $i \neq i'$, are of order $O(n^{-1})$. By Taylor series expansion of loss function $d(\cdot)$ in the neighborhood of $0$, we obtain
         \begin{align*}
            d(z) &\approx d(0) + d'(0)z+ M z^2+1/2(d''(\bar{z})-d''(0))z^2 = Mz^2+R,
         \end{align*}
         where $d(0)=0$, $d'(0)=0$, $M=d''(0)/2! \in (0,\infty)$ and $\bar{z}$ lies between $0$ and $z$.  Assumption \ref{as:5-d} implies $R \le Cz^3$. Therefore
         \begin{align}
           \sum_{i=1}^n d\left\{\sum_{j=1}^n e_{ij}Y_j\right\} &= M\sum_{i=1}^n \left( \sum_{j=1}^n e_{ij}\epsilon_j \right)^2 + O_p\left(1+\sum_{k=1}^d nh_k^{4(p_k+1)}\right) + \sum_{i=1}^n R_i,
                                                                      \label{eqn:th2-qn-dec}
         \end{align}
         where each $R_i \le C|\sum_{j=1}^n e_{ij}\epsilon_j|^3$. The idea is to
         show that the first term in (\ref{eqn:th2-qn-dec}) converges to normal distribution
         and the third term is of smaller order. Using the relation $E|x|^3 \le
         [E|x|^4]^{3/4}$, we obtain
         \begin{align}
           \sum_{i=1}^nR_i \le C \sum_{i=1}^nE|\sum_{j=1}^n e_{ij}\epsilon_j|^3 &\le C \sum_{i=1}^n \left\{ E\left\vert \sum_{j=1}^n e_{ij}\epsilon_j \right\vert^4 \right\}^{3/4} \nonumber \\
                                                                               & \le CC^*  \sum_{i=1}^n \left\{ \sum_{j=1}^n e_{ij}^4E[\epsilon_j ]^4 \right\}^{3/4}+ CC^* \sum_{i=1}^n \left\{  \sum_{j\ne j'}^n   e_{ij}^2 e_{ij'}^2 E[\epsilon_j \epsilon_{j'}]^2 \right\}^{3/4} \nonumber \\
                                                                               &= O_p\left( n (1/n^3h_d^3) + n(n(n-1)/n^4h_d^2)^{3/4} \right) \nonumber\\
           &=O_p(n^{-2}h_d^{-3}) + O_p(1/h_d\sqrt{nh_d})= O_p(1/h_d\sqrt{nh_d}), \label{eqn:th2-ri} 
         \end{align} 
         where $C^*$ is some positive constant and the exact value of it can be calculated using the expression in page 101 of \citet{lin2010probability}.
         Note that, the first term in (\ref{eqn:th2-qn-dec}) can be written as
         \begin{align}
           M \sum_{i=1}^n \left( \sum_{j=1}^n e_{ij}\epsilon_j \right)^2 &= M\sum_{i=1}^n \sum_{j=1}^n e_{ij}^2 \epsilon_j^2 + M \sum_{i=1}^n \sum_{j\neq j'}^n e_{ij} e_{ij'}\epsilon_j \epsilon_{j'} = T_{n1}+T_{n2}. \label{eqn:th2-tn1n2} 
         \end{align}
         After some algebra, we obtain
         \begin{align*}
          \sum_{i , j} e_{ij}^2=\frac{| \Omega_d |}{h_d} \int \left\{ \sum_{l=0}^{p_d}\sum_{m=0}^{p_d} (K_l*K_m)(u) (-1)^m s^{(m+1),(l+1)} \right\}^2 du + o(h_d^{-1}).
          \end{align*}
          Application of Chebyshev inequality yields that $T_{n1}= M\sigma^2\nu_n+O_p(1/\sqrt{n }h_d)$ where $\nu_n=E(\sum_{i,j}e_{ij}^2)$. Now it remains to show that $T_{n2}$ converges to normal in distribution. Observe that $E(T_{n2})=0$ and
\begin{align}
  var(T_{n2}| \mathcal{X})  &= M^2 \sigma^4 \sum_{j \neq j'} \left( \sum_{i=1}^n e_{ij}e_{ij'} \right)^2 = M^2 \sigma^4 \sum_{j \neq j'} (\bm{e}_{j}^T\bm{e}_{j'})^2= M^2 \sigma^4 \delta_{n}^2 \nonumber,
\end{align}
where $\bm{e}_k=(e_{1k},\ldots,e_{nk})^T$ and $\delta_{n}^2 = \sum_{j \neq j'}(\bm{e}_j^T\bm{e}_{j'})^2$. We note that the leading terms of 
\begin{align*}
  \delta_{n}^2  &= \frac{| \Omega_d |}{h_d} \int \bigg\{ \int \bigg[ \sum_{l=0}^{p_d} \sum_{m=0}^{p_d} (K_l * K_m) (u+v) (-1)^m s^{(m+1),(l+1)} \bigg] \\ & \qquad \times \bigg[ \sum_{l=0}^{p_d} \sum_{m=0}^{p_d} (K_l * K_m) (v) (-1)^m s^{(m+1),(l+1)} \bigg] dv \bigg\}^2 du + o_p(h_d^{-1}).
\end{align*}
Therefore, application of Proposition 3.2 of \citet{de1987central} yields
         \begin{align}
           \frac{1}{M\sigma^2} \delta_n^{-1} T_{n2} |\mathcal{X} \xrightarrow{d} N(0,1).
           \label{eqn:th2-clt}
           \end{align}
By plugging  (\ref{eqn:th2-ri}) and (\ref{eqn:th2-tn1n2}) in
         (\ref{eqn:th2-qn-dec}), we obtain
         \begin{align*}
           Q_n &= T_{n1}+T_{n2}+\sum_{i=1}^n R_i+O_p\left(1+\sum_{k=1}^d nh_k^{4(p_k+1)}\right).
         \end{align*}
         Since $n^{-1}RSS_1 =\sigma^2+o_p(1)$ and $q_n=Q_n/n^{-1}RSS_1$, we have,
         \begin{align}
           \frac{q_{n}(H_0)}{M}-\nu_n-\frac{b_{1n}}{\sigma^2}+o_p(h_d^{-1}) \approxeq \frac{T_{n2}}{M\sigma^2},
           \label{eqn:th2-qn-adj}
         \end{align}
         where $b_{1n}=O_p\left(1+\sum_{k=1}^d nh_k^{4(p_k+1)}\right)$.
Therefore, combination of (\ref{eqn:th2-clt}) and (\ref{eqn:th2-qn-adj})  yields
            \begin{align*}
              P\left\{ \delta_n^{-1} \left(  \frac{q_n(H_0)}{M}-\nu_n-\frac{1}{\sigma^2}b_{1n} \right) < t |\mathcal{X} \right\} \xrightarrow{d} \Phi(t).
            \end{align*}
If $nh_k^{4(p_k+1)}h_d \rightarrow 0$ for $k=1,\ldots,d$, then
$b_{1n}=o_p(h_d^{-1})$ which is dominated by $\nu_n$. Then $s_k M^{-1}q_n(H_0)|\mathcal{X} \xrightarrow{} \chi^2_{s_k\nu_n}$ as $n \rightarrow \infty$.            

\vspace{1em}

\underline{\textbf{Proof of (\ref{eqn:th2-f}):}}
Recall            
\begin{align*}
  q_n(H_0) &= \frac{Q_n}{n^{-1}SSR1}=\frac{\sum_{i=1}^n d\left\{\sum_{j=1}^n e_{ij}Y_j\right\}}{n^{-1} RSS_1}.
\end{align*}
By Taylor's expansion, as in part(a), the numerator can be written as
$M\bm{y}^T\bm{E}^T\bm{E}\bm{y}+R_n$ where $R_n$ is the remainder
term which is of order $o_p(h_d^{-1})$. As in part (b) of Theorem
  \ref{thm:glr1}, we show that $\bm{E}^T\bm{E}$ is asymptotically an idempotent matrix  and $\bm{E}^T\bm{E}$ and $\bm{D}$ are asymptotically orthogonal. 
Hence, the LFT statistic is
         \begin{align}
           q_n(H_0)  \approxeq \frac{M \bm{y}^T\bm{E}^T\bm{E}\bm{y}}{n^{-1}\bm{y}^T\bm{D}\bm{y}} = F \frac{Mtr(\bm{E}^T\bm{E})}{n^{-1}tr(\bm{D})},
         \end{align}
         which implies that, as in part(b) of Theorem \ref{thm:glr1},
         we have
         \begin{align*}
           \frac{q_n(H_0)}{M \text{tr}(\bm{E}^T\bm{E})} \xrightarrow{} F_{\text{tr}(\bm{E}^T\bm{E}),\text{tr}(\bm{D})},
         \end{align*}
         as $nh_d \rightarrow \infty$ and $n \rightarrow \infty$.        
    
\end{proof}
        


 For the following theorem, we consider the contiguous alternative of the form
 \begin{align}
   H_{1n}:m_d(X_d) &= M_n(X_d),
  \label{eqn:h1a}
 \end{align}
 where $M_n(X_d) \rightarrow 0$ and $M_n \in \mathcal{M}_n(\rho; \eta)$.
          
        
 \begin{theorem} \label{thm:glr-lft-h1}
            Suppose $E\{M_n(X_d)|X_1,\ldots,X_{d-1}\}=0$ and $h_d\cdot\sum_{i=1}^n M_n^2(X_{id}) \xrightarrow{P} C_M$ for some constant $C_M$. Suppose $0 \le p_j \le 3$, for $j=1,\ldots,d$.
            \begin{enumerate}[label=(\roman*)]
              \item{[GLR test]} Suppose that conditions \ref{as:1-x}-\ref{as:4-e} hold. Under $H_{1n}$ for the testing problem (\ref{eqn:h1}),
              \begin{align*}
              P\left\{\sigma_{n}^{-1}\left(\lambda_n(H_0)-\mu_n-\frac{d_{1n}+d_{2n}}{2\sigma^2}\right)<t|\mathcal{X}\right\} \xrightarrow{d} \bm{\Phi}(t),
              \end{align*}
              where $\mu_n$, $d_{1n}$ and $\sigma_n$ are same as those in Theorem \ref{thm:glr1} and
              \begin{align*}
              d_{2n} &= \sum_{i=1}^n M_n^2(X_{id}) (1+o_p(1)).
              \end{align*}
              \item{[LF test]} Suppose that conditions \ref{as:1-x}-\ref{as:5-d} hold. Under $H_{1n}$ for the testing problem (\ref{eqn:h1}),
              \begin{align*}
              P\left\{\delta_{n}^{-1}\left(\frac{q_n(H_0)}{M}-\nu_n-\frac{b_{1n}+b_{2n}}{\sigma^2}\right)<t|\mathcal{X}\right\} \xrightarrow{d} \bm{\Phi}(t),
              \end{align*}
              where $\nu_n$, $b_{1n}$ and $\delta_{n}$ are same as those in Theorem \ref{thm:loss1} and 
              \begin{align*}
              b_{2n} &= \sum_{i=1}^n M_n^2(X_{id}) (1+o_p(1)). 
              \end{align*} 
            \end{enumerate}
            
\end{theorem}

\begin{proof}  \underline{\textbf{Part (i):}}
          	Under $H_{1n}$ (\ref{eqn:h1a}), the arguments analogous to Lemma \ref{lem:an12} yields,  
          	\begin{align*}
          	(\bm{I}-\bm{W})\bm{m}_+  &=  (\bm{G}^{\perp} -\sum_{j=1}^d \bm{H}_{p_j,j}^* \bm{G}^{\perp}) (\bm{m}_1+\ldots+ \bm{m}_{d-1}+\bm{M}_n) + o_p(1)\\
          	&= \bm{1}\cdot O\left(\sum_{k=1}^d h_k^{2(p_k+1)}\right) + \bm{1}\cdot O_p\left(1/\sqrt{n}\right),
          	\end{align*}
          	where $\bm{M}_n(\bm{x}_d)=(M_n(X_{1d}),\ldots,M_n(X_{nd}))^T$, $\bm{M}_n \in \mathcal{M}_n(\rho; \eta)$ defined in (\ref{eqn:class-h1}) and $\bm{1}$ is the vector of ones of size $n$. Similarly,
          	\begin{gather*}
          	(\bm{I}-\bm{W}^{[-d]})\bm{m}_+  = \left( \bm{G}_{[-d]}^{\perp} -\sum_{j=1}^{d-1} \bm{H}_{p_j,j}^* \bm{G}_{[-d]}^{\perp}  \right)\bm{M}_n +\bm{1}\cdot O\left(\sum_{k=1}^{d-1} h_k^{2(p_k+1)}\right) + \bm{1}\cdot O_p\left(1/\sqrt{n}\right).
            \end{gather*}
            Observe that, same set of arguments yield $(\bm{I}-\bm{W})\bm{\epsilon} = \bm{\epsilon} + o_p(1)$   and    $(\bm{I}-\bm{W}^{[-d]})\bm{\epsilon} =  \bm{\epsilon} +o_p(1)$.
          	Consider,
          	   	\begin{eqnarray}
          	   	RSS_0-RSS_1 &=& \bm{y}^T(A_{n1}-A_{n2})\bm{y} \nonumber \\
          	   	&=& \bm{\epsilon}^T(A_{n1}-A_{n2})\bm{\epsilon} +  \bm{m}_+^T(A_{n1}-A_{n2})\bm{m}_+ + 2\bm{\epsilon}^T(A_{n1}-A_{n2})\bm{m}_+  \nonumber\\
          	   	&=&I_{n1}+I_{n2}+I_{n3}.
          	   	\label{eqn:th3-rs01}
          	   	\end{eqnarray}
          	straightforward computations  yield
          	\begin{align*}
          	I_{n2} &= \bm{M}_n^T\bm{M}_n + O_p\left(1+\sum_{k=1}^d nh_k^{4(p_k+1)}\right) \qquad \text{ and }\\
          	I_{n3} &= \bm{\epsilon}^T\bm{M}_n +O_p\left(1+ \sum_{k=1}^{d} \sqrt{n}h_{k}^{2(p_k+1)}\right).
          	\end{align*}
          	Plugging the above results and the $I_{n1}$ value from Theorem \ref{thm:glr1} in (\ref{eqn:th3-rs01}), we obtain
          	\begin{eqnarray}
          	RSS_0-RSS_1 &=& L_1+L_2+ C_n+ d_{2n} +d_{1n},
          	\end{eqnarray}
          where $L_1$, $L_2$, $d_{1n}$ are same as defined in Theorem \ref{thm:glr1},
          	\begin{align*}
          	d_{2n} &=  \sum_{i=1}^n M_n^2(X_{id})+o_p(h_d^{-1}) \qquad \text{ and } \\
          	C_n &=  \sum_{i=1}^n
          	\epsilon_i M_n(X_{id}).
          	\end{align*}
          	The proof follows by proceeding similar to Theorem \ref{thm:glr1}.
          
 
 \underline{\textbf{Part (ii):}}
The arguments analogous to Lemma \ref{lem:an12} yield that, under $H_1$,
	\begin{align*}
	(\bm{W}-\bm{W}^{[-d]})\bm{m}_+  &=  (\bm{P}_{\bm{G}_{[-d]}^{\perp}\mathbb{X}_d^{[-0]}}+\bm{H}_{p_d,d}^*\bm{G}^{\perp}) (\bm{m}_1+\ldots+\bm{m}_{d-1}+\bm{M}_n)+ o_p(1) \\ &=\bm{M}_n+\bm{1}\cdot O\left(\sum_{k=1}^d h_k^{2(p_k+1)}\right) + \bm{1}\cdot O_p\left(1/\sqrt{n}\right).
	\end{align*}
	Similarly, 
$(\bm{W}-\bm{W}^{[-d]})\bm{\epsilon} =(\sum_{j=1}^ne_{1j}\epsilon_j,\ldots,\sum_{j=1}^n e_{nj}\epsilon_j)^T$
	where $e_{ij}$ is the $(i,j)$th $1 \le i,j \le n$, element in the matrix $ \left\{\bm{P}_{\bm{G}_{[-d]}^{\perp}\mathbb{X}_d^{[-0]}}+\bm{H}_{p_d,d}^*\bm{G}^{\perp} + O(n^{-1}h_d^{-1}\bm{I}+ n^{-1}\bm{J}) \right\}$. Note that the leading terms of  $e_{ij}$'s are of the same order as the elements in  $\bm{H}_{p_d,d}^*$. By Taylor expansion,
	\begin{align}
    \sum_{i=1}^n d\left\{\sum_{j=1}^n e_{ij}Y_j\right\} &= M\sum_{i=1}^n \left( \sum_{j=1}^n e_{ij}\epsilon_j \right)^2 +M \sum_{i=1}^n M_n^2(X_{id}) \nonumber \\& +O_p\left(1+\sum_{k=1}^d nh_k^{4(p_k+1)}\right) + \sum_{i=1}^n R_i,
	\label{eqn:th4-qn-dec}
	\end{align}
	where each $R_i \le C|\sum_{j=1}^n e_{ij}\epsilon_j|^3$. The  proof follows by proceeding similar to Theorem \ref{thm:loss1}.
\end{proof}


\begin{theopargself}
  \begin{theorem}\label{thm:opt-test}
    Under conditions \ref{as:1-x}-\ref{as:5-d}, if
    $h_k^{2(p_k+1)}=O(h_d^{2(p_d+1)})$ and $0 \le p_k \le 3$, for $k=1,\ldots,d-1$, then for the testing
    problem (\ref{eqn:h1}), both GLR and LF tests can detect alternatives
    with rate  $\rho_n=n^{-\frac{4(p_d+1)}{8p_d+9}}$ when $h_d=c_*n^{-\frac{2}{8p_d+9}}$ for some
    constant $c_*$. 
  \end{theorem}
\end{theopargself}

  \begin{proof}
     The proof uses arguments analogous to Theorem 5 in \citet{fan2005nonparametric}. We provide proof only for the GLR
     test and similar arguments can be used to prove the LF test.  Under the contiguous alternative $H_{1n}:m_d(X_d)=M_n(X_d)$, it follows from (i) of Theorem \ref{thm:glr-lft-h1}, 
              \begin{align}
                \lambda_n(H_0) &= \mu_n +\frac{L_2}{2\sigma^2}+\frac{d_{2n}+C_n}{2\sigma^2}+O_p\left(1+ \sum_{k=1}^dnh_k^{4(p_k+1)}+\sum_{k=1}^d\sqrt{n}h_k^{2(p_k+1)}\right),
              \end{align}
              where $d_{2n}=\sum_{i=1}^n M_n^2(X_{id})$ and $C_n=\sum_{i=1}^n
              \epsilon_i M_n(X_{id})$. Since the probability of the type II
              error at $H_{1n}$ is defined as
              $\beta(\alpha,M_n)=P(\phi_h=0|m_d=M_n)$, it implies that
              \begin{align*}
                \beta(\alpha,M_n) &= P\{\sigma_n^{-1}\left( -\lambda_n(H_0) +\mu_n\right) \ge z_{\alpha} |\mathcal{X} \}\\
                                  &= P \left\{ \sigma_n^{-1}\left( -\frac{L_2}{2\sigma^2}-\frac{d_{2n}+C_n}{2\sigma^2}+O_p\left(1+ \sum_{k=1}^dnh_k^{4(p_k+1)}+\sum_{k=1}^d\sqrt{n}h_k^{2(p_k+1)}\right) \right) \ge z_{\alpha} |\mathcal{X} \right\} \\
                &= P_{1n}+P_{2n},
              \end{align*}
              with
              \begin{align*}
                P_{1n} &= P\left\{ \sigma_n^{-1}(-\frac{L_2}{2\sigma^2})+\sqrt{n}h_d^{(4p_d+5)/2}t_{1n}+nh_d^{(8p_d+9)/2}t_{2n}-\sqrt{h_d}t_{3n} \ge z_{\alpha}, |t_{1n}|\le M, |t_{2n}| \le M |  \mathcal{X} \right\},\\ P_{2n} &= P\left\{ \sigma_n^{-1}(-\frac{L_2}{2\sigma^2})+\sqrt{n}h_d^{(4p_d+5)/2}t_{1n}+nh_d^{(8p_d+9)/2}t_{2n}-\sqrt{h_d}t_{3n} \ge z_{\alpha}, |t_{1n}|\ge M, |t_{2n}|\ge M | \mathcal{X} \right\},
              \end{align*}
              and
              \begin{align*}
                t_{1n} &= \left( \sqrt{n}h_d^{(4p_d+5)/2}\sigma_n\right)^{-1} O_p\left(1+\sum_{k=1}^d\sqrt{n}h_k^{2(p_k+1)}\right)= O_p(1),\\
                t_{2n} &= \left( nh_d^{(8p_d+9)/2} \sigma_n\right)^{-1} O_p(\sum_{k=1}^dnh_k^{4(p_k+1)})= O_p(1),\\
                 t_{3n} &= (\sqrt{h_d}\sigma^2 \sigma_n)^{-1} \frac{1}{2}[d_{2n}+C_n]. 
              \end{align*}
             Note that $E[C_n|\mathcal{X}]=0$ and
             $var(C_n|\mathcal{X})= O(\sum_{i=1}^nM_n^2(X_{id}))$ and hence
             $C_n=O_p(\sqrt{d_{2n}})$. Analogous arguments to Lemma B.7 of
             \citet{fan2005nonparametric} lead to 
             \begin{align*}
               \sqrt{h_d} t_{3n} \rightarrow \infty \qquad \text{ only when  } \qquad n\sqrt{h_d}\rho^2 \rightarrow \infty.  
             \end{align*}
             We choose $h_d \le c_0^{-\frac{1}{2(p_d+1)}}n^{-\frac{1}{4(p_d+1)}}$.
             This implies, $\sqrt{n}h_d^{(4p_d+5)/2} \ge c_0 nh_d^{(8p_d+9)/2}$,
             $\sqrt{n}h_d^{(4p_d+5)/2} \rightarrow 0$, and $nh_d^{(8p_d+9)/2}
             \rightarrow 0$. Hence, for $h_d \rightarrow 0$ and $nh_d\rightarrow
             \infty$, it follows that $\beta(\alpha,\rho) \rightarrow 0$ only
             when $n\sqrt{h_d}\rho^2 \rightarrow +\infty$. This implies $\rho_n^2 =
 n^{-1}h_d^{-1/2}$ and the possible minimum value of $\rho_n$ in this setting is  
$n^{\frac{-(8p_d+7)}{16(p_d+1)}}$. When $nh_d^{4(p_d+1)} \rightarrow \infty$, for any $\delta>0$, there
exists a constant $M>0$ such that $P_{2n}<\frac{\delta}{2}$ uniformly in $M_n
\in \mathcal{M}_n(\rho; \eta)$. Then
\begin{align*}
  \beta(\alpha, \rho) \le \frac{\delta}{2} + P_{1n}.
\end{align*}
We note that $\underset{\mathcal{M}_n(\rho; \eta)}{\sup} P_{1n} \rightarrow 0$ only
when $B(h_d)\equiv nh_d^{(8p_d+9)/2}M-nh_d^{1/2}\rho^2 \rightarrow -\infty$. The function
$B(h_d)$ attains the  minimum value
\begin{align*}
-\frac{8(p_d+1)}{8p_d+9}[M (8p_d+9)]^{-\frac{1}{8(p_d+1)}}n \rho^{\frac{8p_d+9}{4(p_d+1)}}
\end{align*} at
$h_d=\left[ \frac{\rho^2}{M(8p_d+9)} \right]^{\frac{1}{4(p_d+1)}}$. With simple
algebra, in this setting, we obtain the corresponding minimum value of
$\rho_n=n^{-\frac{4(p_d+1)}{8p_d+9}}$ at $h_d=c_*n^{-\frac{2}{8p_d+9}}$ for some
constant $c_*$.
\end{proof}


\begin{theorem}
  \label{thm:are}[Relative efficiency] Suppose Conditions \ref{as:1-x}--\ref{as:5-d} hold, $h \propto n^{-\omega}$ for $\omega \in (0,1/(4p_d+5))$ and $p_j = 0$ for $j=1,\ldots,d$. Then Pitman's relative efficiency of the LF test over the GLR test under  $H_n$ in (\ref{eqn:h1-spl}) is given by
  \begin{align*}
    \text{ARE}&(q_n,\lambda_n)  \\
    &=  \left[ \frac{\int \left\{2(K_0*K_0)(u) -\int (K_0*K_0)(u+v) (K_0*K_0)(v) dv \right\}^2 du}{\int \left\{\int (K_0*K_0)(u+v) (K_0*K_0)(v) dv \right\}^2 du} \right]^{1/(2-3\omega)}.
    \end{align*}
    The asymptotic relative efficiency ARE$(q_n,\lambda_n)$ is larger than 1 for any kernel satisfying Condition \ref{as:2-ker-h} and $K(\cdot) \le 1$.
\end{theorem}

\begin{proof}
  Pitman's asymptotic relative efficiency of the LF test over the GLR test is the limit of the ratio of the sample sizes required by the two tests to have the same asymptotic power at the same significance level, under the same local alternative [\citet{pitman2018some}, Chapter 7]. Suppose $n_1$ and $n_2$ are the sample sizes required for the LF test and the GLR test, respectively. The Pitman's asymptotic relative efficiency of $q_n$ to $\lambda_n$ is defined as
  \begin{align*}
    \text{ARE}(q_n,\lambda_n) &= \lim_{n_1,n_2 \rightarrow \infty} \frac{n_1}{n_2},
  \end{align*}
  under the condition that $\lambda_n$ and $q_n$ have the same asymptotic power under the same local alternatives $n_1^{-1/2} h_{d_1}^{-1/2} g_1(x_d) \sim n_2^{-1/2} h_{d_2}^{-1/2} g_2(x_d)$ in the sense that
\begin{align*}
  \lim_{n_1,n_2 \rightarrow \infty} \frac{n_1^{-1/2}h_{d_1}^{-1/2}g_1(x_d)}{n_2^{-1/2}h_{d_2}^{-1/2}g_2(x_d)} &=1.
\end{align*}
Given $h_{d_i} =c n_i^{-\omega}$, $i=1,2$, we have $n_1^{-2\gamma} \sum_{i=1}^n g_1^2(X_{di}) \sim n_2^{-2\gamma} \sum_{i=1}^n g_2^2(X_{di})$, where $\gamma=(1-\omega)/2$. Hence,
\begin{align}
  \lim_{n_1,n_2 \rightarrow \infty} \left( \frac{n_1}{n_2} \right)^{2\gamma} &= \frac{\sum_{i=1}^n g_1^2(X_{di})}{\sum_{i=1}^n g_2^2(X_{di})}.
  \label{thm-are:c1}
\end{align}
From Theorem \ref{thm:glr-lft-h1}(i), we have
\begin{align*}
  \frac{\lambda_{n_1}(H_0)-\mu_{n_1}}{\sigma_{n_1}} \xrightarrow[]{d} N(\xi,1),
\end{align*}
under $H_{n_1}: m_d(x_d)=n_1^{-1/2}h_{d_1}^{-1/2}g_1(x_d)$, where $\xi= [\sum_{i=1}^n g_1^2(X_{di})]/(2\sigma^2 \sigma_{n_1})$ with $\sigma_{n_1}$ is defined in Theorem \ref{thm:glr1}. Also, from Theorem \ref{thm:glr-lft-h1}(ii), we have
\begin{align*}
  \frac{M^{-1}q_{n_2}(H_0)-\nu_{n_2}}{\delta_{n_2}} \xrightarrow[]{d} N(\psi,1),
\end{align*}
under $H_{n_2}: m_d(x_d)=n_2^{-1/2}h_{d_2}^{-1/2}g_2(x_d)$, where $\psi=  [\sum_{i=1}^n g_2^2(X_{di})]/(\sigma^2 \delta_{n_2})$ with $\delta_{n_2}$ is defined in Theorem \ref{thm:loss1}. To have the same asymptotic power, the noncentrality parameters must be equal which means $\xi=\psi$ or
\begin{align}
\frac{\sum_{i=1}^n g_1^2(X_{di})}{\sum_{i=1}^n g_2^2(X_{di})} =  \frac{2\sigma_{n_1}}{\delta_{n_2}}. 
\label{thm-are:c2}
\end{align}
Combination of (\ref{thm-are:c1}) and (\ref{thm-are:c2}) yields, for $p_j=0$, $j=1,\ldots,d$,
\begin{align*}
\text{ARE}(q_n,\lambda_n) &= \left[ \frac{2h_{d_1}^{1/2}\sigma_{n_1}}{h_{d_2}^{1/2}\delta_{n_2}} \right]^{2/(2-3\omega)}= \left[ \frac{4h_{d_1}\sigma_{n_1}^2}{h_{d_2}\delta_{n_2}^2} \right]^{1/(2-3\omega)}\\
&=  \left[ \frac{\int \left\{2(K_0*K_0)(u) -\int (K_0*K_0)(u+v) (K_0*K_0)(v) dv \right\}^2 du}{\int \left\{\int (K_0*K_0)(u+v) (K_0*K_0)(v) dv \right\}^2 du} \right]^{1/(2-3\omega)}.
\end{align*}
Now, we show $\text{ARE}(q_n,\lambda_n) \ge 1$ for any positive kernels with $K(\cdot) \le 1$. It is sufficient to show that 
\begin{align*}
  \int \bigg\{2(K_0*K_0)(u) &-\int (K_0*K_0)(u+v) (K_0*K_0)(v) dv \bigg\}^2 du \\ & \qquad \ge \int \left\{\int (K_0*K_0)(u+v) (K_0*K_0)(v) dv \right\}^2 du.
\end{align*}
From Jensen's inequality and Fubini's theorem we obtain
\begin{align}
  \int \bigg\{\int & (K_0*K_0)(u+v) (K_0*K_0)(v) dv \bigg\}^2 du \nonumber \\ & \qquad \le \int \int (K_0*K_0)^2(u+v) (K_0*K_0)(v) dv du \nonumber \\
  & \qquad = \int (K_0*K_0)^2(u) du. \label{thm-are:fu} 
\end{align}
Triangle inequality and (\ref{thm-are:fu}) yields that
\begin{align}
  \bigg\{ & \int  \bigg\{2(K_0*K_0)(u) -\int (K_0*K_0)(u+v) (K_0*K_0)(v) dv \bigg\}^2 du \bigg\}^{1/2} \nonumber \\
  &\qquad \ge  2 \bigg\{ \int (K_0*K_0)^2(u) du \bigg\}^{1/2}- \bigg\{ \int \bigg\{ \int (K_0*K_0)(u+v) (K_0*K_0)(v) dv \bigg\}^2 du \bigg\}^{1/2} \nonumber \\
  &\qquad \ge 2 \bigg\{ \int (K_0*K_0)^2(u) du \bigg\}^{1/2}- \bigg\{ \int (K_0*K_0)^2(u) du \bigg\}^{1/2} \nonumber \\
  &\qquad = \bigg\{ \int (K_0*K_0)^2(u) du \bigg\}^{1/2}. \label{thm-are:te}
\end{align}
Combination of (\ref{thm-are:te}) and (\ref{thm-are:fu}) yields
\begin{align*}
  \bigg\Vert & 2(K_0*K_0)(u) -\int (K_0*K_0)(u+v) (K_0*K_0)(v) dv \bigg\Vert_2 \\
  & \qquad  \ge \bigg\Vert (K_0*K_0)(u) \bigg\Vert_2 \ge \bigg\Vert \int (K_0*K_0)(u+v) (K_0*K_0)(v) dv \bigg\Vert_2. 
\end{align*} 
Hence, the LF test is asymptotically more efficient than the GLR test.

\end{proof}

\section*{Numerical Comparison- Extra results}

\subsection*{Conditional Bootstrap}
\begin{enumerate}[label=(\alph*)]
  \item Fix the bandwidths at their estimated values $(\widehat{h}_1, \widehat{h}_2,\widehat{h}_3, \widehat{h}_4)$ and then obtain the estimates of additive functions under both null and unrestricted additive models.
  \item Compute  $\lambda_n$, $q_n$, $\lambda_n(FJ)$, $F_{\lambda}$, $F_q$, $S_n$  and the residuals $\widehat{\epsilon}_i$, $i=1,\ldots,n$, from the unrestricted model.
  \item For each $(X_{1i},X_{2i},X_{3i},X_{4i})$, draw a bootstrap residual $\widehat{\epsilon}_i^*$ from the centered empirical distribution of $\widehat{\epsilon}_i$ and compute $Y_i^*= \widehat{m}_0+ \widehat{m}_1(X_{1i})+ \widehat{m}_3(X_{3i})+\widehat{m}_4(X_{4i})+ \widehat{\epsilon}_i^*$, where $\widehat{m}_1$, $\widehat{m}_3$ and $\widehat{m}_4$ are the estimated additive functions under the restricted model in step (a). This forms a conditional bootstrap sample $(Y_i^*, X_{1i},X_{2i},X_{3i},X_{4i})_{i=1}^n$.
  \item Using the bootstrap sample in step (c) and bandwidths in step (a), obtain  $\lambda_n^*$, $q_n^*$, $\lambda_n^*(FJ)$, $F^*_{\lambda}$, $F^*_q$, $S_n^*$.
  \item Repeat steps (c) and (d) for a total of $B$ times, where $B$ is large number. We then obtain a sample of statistics.
  \item Compute the bootstrap $P$ values $P_{\lambda}^*= B^{-1}\sum_{l=1}^B \bm{1}(\lambda_n < \lambda_{nl}^*)$ for all the statistics. Reject $H_0$ at a prespecified significance level $\alpha$ if and only if $P_{\lambda}^* <\alpha$. Repeat this process for the all the above statistics. 
\end{enumerate}


 %

%
\begin{figure}[!htbp]
  \centering
 \includegraphics[scale=0.3]{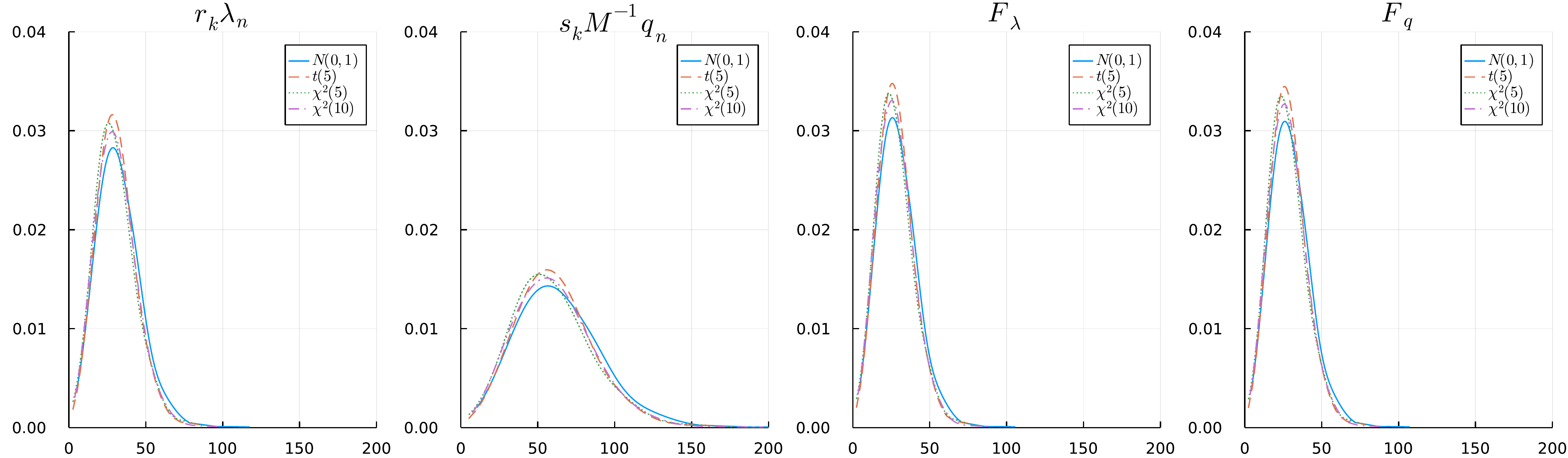}
   \caption{Estimated densities for scaled GLR  and   LF test statistics, and F statistics, among 1000 simulations under different errors (--- normal; $- - - t(5)$; $\cdots \chi^2(5)$; $-\cdot- \chi^2(10)$ ). Here, the errors except normal are scaled to have mean 0 and variance 1. }
   \label{fig:sim1-1b}
\end{figure}

\begin{figure}
  \centering
  \includegraphics[scale=0.35]{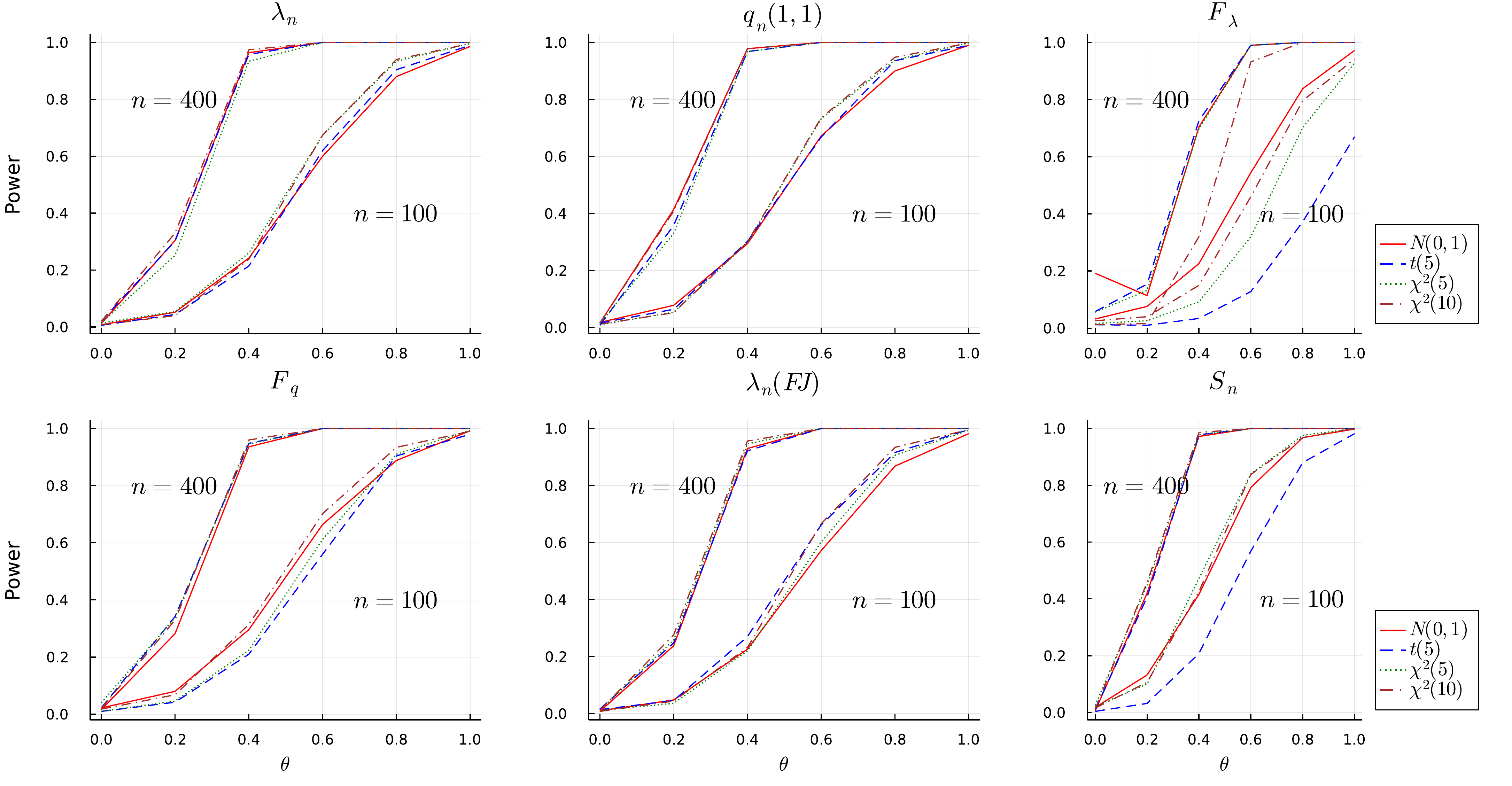}
  \caption{Power of the tests under alternative model sequence (\ref{eqn:htheta}) using bandwidths $S_X n^{-2/17}$ at 1\% level of significance. Only the LF test with LINEX loss function (\ref{eqn:linex-def}) for $s=1,t=1$ is reported. The power values are similar for other choices of $s$ and $t$. }
  \label{fig:sim-one-bseven}
\end{figure}
%

\begin{figure}
  \centering
  \includegraphics[scale=0.35]{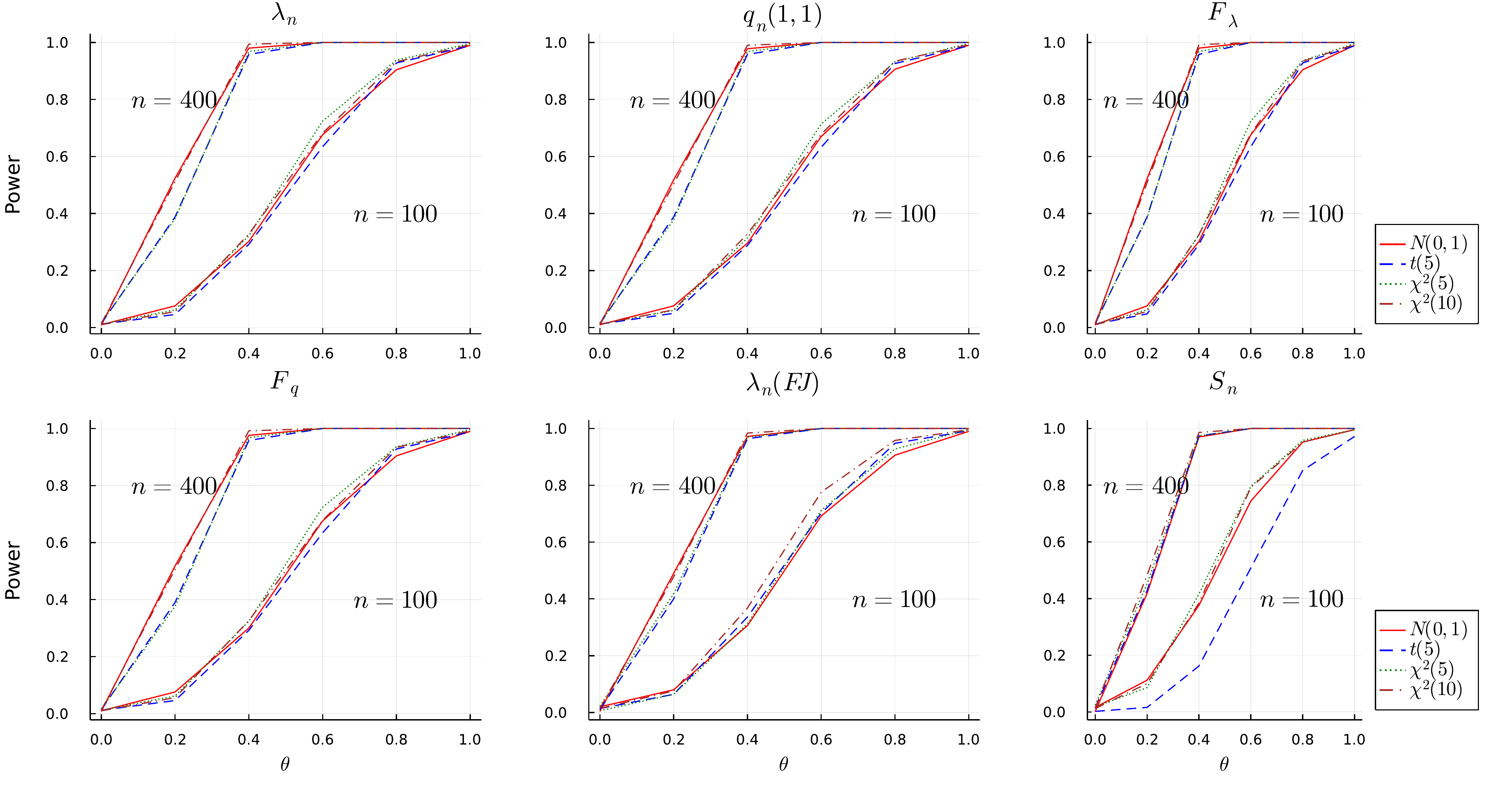}
  \caption{Power of the tests under alternative model sequence (\ref{eqn:htheta}) using optimal bandwidths (cross-validation) at 1\% level of significance. Only the LF test with LINEX loss function (\ref{eqn:linex-def}) for $s=1,t=1$ is reported. The power values are similar for other choices of $s$ and $t$. }
  \label{fig:sim-one-bopt}
\end{figure}

\end{document}